\documentclass[12pt, a4paper,reqno]{amsart}

\usepackage{color} \definecolor{bleu_sombre}{rgb}{0,0,0.6}  \definecolor{rouge_sombre}{rgb}{0.8,0,0}\definecolor{vert_sombre}{rgb}{0,0.6,0}
\usepackage[plainpages=false,colorlinks,linkcolor=bleu_sombre,citecolor=rouge_sombre,urlcolor=vert_sombre,breaklinks]{hyperref}
\usepackage{hyperref}

\usepackage[english]{babel}

\usepackage{amsmath,amssymb,amsthm,graphicx,amsfonts,url,color,enumerate,dsfont,stmaryrd,mathabx}

\usepackage[normalem]{ulem}

\theoremstyle{plain}
\newtheorem{theorem}{{Theorem}}[section]
\newtheorem*{theorem*}{{Theorem}}
\newtheorem{proposition}[theorem]{Proposition}
\newtheorem*{proposition*}{Proposition}

\newtheorem*{corollary*}{Corollary}
\newtheorem{lemma}[theorem]{Lemma}
\newtheorem*{lemma*}{Lemma}

\theoremstyle{definition}
\newtheorem{definition}[theorem]{Definition}
\newtheorem*{definition*}{Definition}

\theoremstyle{remark}
\newtheorem{remark}[theorem]{Remark}

\makeatletter

\@addtoreset{equation}{section}  
\makeatother

\newcommand {\limt}[2]{\xrightarrow[#1 \to #2]{}}

\newcommand{\abs}[1]{\left\vert #1\right\vert}        
\newcommand{\nr}[1]{\left\Vert #1\right\Vert}         
 
\newcommand{\innp}[2]{\left< #1 , #2 \right>}         

\newcommand{\set}[1]{\left\{ #1 \right\}}		
\newcommand{\Ii}[2] {\left\{ #1,\dots,#2 \right\}}
\renewcommand{\leq}{\leqslant}	\renewcommand{\geq}{\geqslant}

\newcommand{\inv}{^{-1}}

\newcommand{\littleo}[2]{\mathop{o}\limits_{#1 \to #2}}
\newcommand{\1}{\mathds 1}
\newcommand{\st}{\,:\,}
         
\renewcommand{\Re}{\mathsf{Re}}        
\renewcommand{\Im}{\mathsf{Im}}  
 
\newcommand{\supp}{\mathsf{supp}}

\newcommand{\Dom}{\mathsf{Dom}}
\newcommand{\Sp}{\mathsf{Sp}}
\renewcommand{\ker}{\mathsf{ker}} 
\newcommand{\Ran}{\mathsf{Ran}}

\newcommand{\R}{\mathbb{R}}		\newcommand{\C}{\mathbb{C}}
\newcommand{\N}{\mathbb{N}}	\newcommand{\Z}{\mathbb{Z}}	
\newcommand{\T}{\mathbb{T}}

\renewcommand{\a}{\alpha}\renewcommand{\b}{\beta}\newcommand{\g}{\gamma}\newcommand{\G}{\Gamma}\renewcommand{\d}{\delta}\newcommand{\e}{\varepsilon}\newcommand{\z}{\zeta} \renewcommand{\th}{\theta}\newcommand{\Th}{\Theta}\renewcommand{\k}{\kappa}\renewcommand{\l}{\lambda}\newcommand{\m}{\mu}\newcommand{\s}{\sigma}\newcommand{\f}{\varphi}\newcommand{\vf}{\phi}\newcommand{\h}{\chi}\renewcommand{\o}{\omega}\renewcommand{\O}{\Omega}

\newcommand{\Hc}{{\mathcal H}}\newcommand{\Lc}{{\mathcal L}}\newcommand{\Oc}{{\mathcal O}}\newcommand{\Tc}{{\mathcal T}}

\newcommand{\stepp}{\noindent {\bf $\bullet$}\quad }

\usepackage{mathrsfs}

\begin{document}

\renewcommand{\Tc}{T_{\mathsf{c}}}
\newcommand{\Tp}{\Tmax}
\newcommand{\Tm}{T_{-}}
\newcommand{\Tmin}{T_{\mathsf{min}}}
\newcommand{\Tmax}{T_{\mathsf{max}}}

\newcommand{\LL}{\mathscr L}

\newcommand{\uo}{u_{\mathsf{o}}} \newcommand{\uon}{u_{\mathsf{o},n}} \newcommand{\uom}{u_{\mathsf{o},m}}

\title{Critical time for the observability of Kolmogorov-type equations}

\author{J\'er\'emi Dard\'e and Julien Royer}

\begin{abstract}
This paper is devoted to the observability of a class of two-di\-men\-sion\-al Kol\-mo\-go\-rov-type equations presenting a quadratic degeneracy. We give lower and upper bounds for the critical time. These bounds coincide in symmetric settings, giving a sharp result in these cases. The proof is based on Carleman estimates and on the spectral properties of a family of non-selfadjoint Schr\"odinger operators, in particular the localization of the first eigenvalue and Agmon type estimates for the corresponding eigenfunctions.
\end{abstract}

\maketitle

\section{Introduction}

This paper is devoted to the study of the observability of two-dimensional Kolmogorov-type equations with a quadratic degeneracy. Let $\ell_+, \ell_- >0$. We set $I = ]-\ell_-,\ell_+[$ and $\O = \T \times I$, where $\T$ is the one-dimensional torus $\R / (2 \pi \Z)$. All along the paper, a generic point in $\O$ will be denoted by $(x,y)$, with $x \in \T$ and $y \in I$.\\

We consider $q \in C^3 (\bar I,\R)$ such that
\[
q(0) = 0 \quad \text{and} \quad  \min_{y \in I} q'(y) > 0.
\]
In particular, $q(y) \neq 0$ for $y \neq 0$. The model case is $q(y) = y$.\\

Then, for $T > 0$, we consider on $\O$ the Kolmogorov-type equation
\begin{equation} \label{eq-Kolmogorov}
\begin{cases}
\partial_t u + q(y)^2 \partial_x u - \partial_{yy} u = 0, & \text{on } ]0,T[ \times \O,\\
u(t,\cdot) = 0, & \text{on } \partial\O, \text{ for all } t \in ]0,T[,\\
u_{|t=0} \in L^2 (\O).
\end{cases}
\end{equation}
We are interested in the observability properties of the problem \eqref{eq-Kolmogorov}:

\begin{definition}
\begin{enumerate}[\rm (i)]
\item We say that \eqref{eq-Kolmogorov} is observable in time $T$ through an open subset $\omega$ of $\O$ if there exists $C > 0$ such that for any solution $u$ of \eqref{eq-Kolmogorov} we have 
\begin{equation} \label{eq-obs-vol}
\nr{u(T)}_{L^2(\O)}^2 \leq C \int_0^T \nr{u(t)}_{L^2(\o)}^2 \, dt.
\end{equation}
\item We say that \eqref{eq-Kolmogorov} is observable in time $T$ through an open subset $\Gamma$ of the boundary $\T \times \{-\ell_-,\ell_+\}$ of $\O$ if there exists $C > 0$ such that for any solution $u$ of \eqref{eq-Kolmogorov} we have 
\begin{equation} \label{eq-obs-Gamma}
\nr{u(T)}_{L^2(\O)}^2 \leq C \int_0^T \nr{\partial_\nu u(t)}_{L^2(\G)} \, dt.
\end{equation}
\end{enumerate}
\end{definition}

Null-controllability and observability properties of non-degenerate parabolic equations have been investigated
for several decades now, since the pioneering works \cite{egorov} and \cite{fattorini_russel} which proved independently the null-controllability of the one-dimensional heat-equation. Then \cite{Lebeau_Robbiano}
 and \cite{Fursikov_Imanuvilov} independently generalized this result in any dimension, showing
 that the heat equation is observable through any (interior or boundary) observation set, in any positive time, in any
 geometrical setting.
 
This is not the case for degenerate parabolic equations, which are a more recent subject of study. These equations
may or may not be observable, depending on the location and the strength of the degeneracy, the geometrical setting, and the time horizon $T$.
 The case of a degeneracy of the equation at the boundary of the domain
 is now fairly well-understood (see \cite{Cannarsa_Martinez_Vancostenoble} and the references therein). In general, this type of degenerate equations are observable for weak degeneracy, and are not when the degeneracy becomes too strong.
 
  In the case of interior
 degeneracy, there is no general theory, and equations are for the moment studied one after another. 
 Interestingly, the known results show that,  for precise strength of the degeneracy, a minimal time
 appears, under which  observability is lost.
  
Among parabolic equations with interior degeneracy, the Grushin  equation is so far the best understood: the two-dimensional case  is now
 almost completely  understood,  and some partial results have been obtained in multi-dimensional settings \cite{Beauchard_Cannarsa_Guglielmi,Beauchard_Miller_Morancey,Koenig17,Beauchard_Darde_Ervedoza,Duprez_Koenig,Allonsius_Boyer_Morancey}. Other equations have also been studied, such as the heat equation on the Heisenberg group \cite{Beauchard_Cannarsa}.
 
 Finally, we highlight that a minimal time condition for observability might also appear for systems of parabolic equations, degenerate or not (see, among others, \cite{AmmarKhodja_Benabdallah_GonzalesBurgos_deTeresa,Duprez,Benabdallah_Boyer_Morancey}),
 for degenerate Schr\"odinger equations \cite{Burq_Sun_19}, and appears naturally for the wave equation (see  \cite{raucht74,bardoslr92}).\\

Regarding the Kolmogorov equation \eqref{eq-Kolmogorov},  observability properties have already been investigated in the 
case $q(y) = y$, that is for the system
\begin{equation} \label{eq-Kolm-ysquared}
\begin{cases}
\partial_t u + y^2 \partial_x u - \partial_{yy} u = 0, & \text{on } ]0,T[ \times \O,\\
u(t,\cdot) = 0, & \text{on } \partial\O, \text{ for all } t \in ]0,T[,\\
u_{|t=0} \in L^2 (\O).
\end{cases}
\end{equation}
It is proved in \cite{Beauchard14} that a critical time $\Tc$ appears for the observability through an open set of the form $\omega = \T\times ]a,b[$ if $0 \notin ]a,b[$:

\begin{theorem}[\cite{Beauchard14}] \label{th-KB}
Let $\omega =  \T \times ]a,b[$ with $-\ell_- < a < b< \ell_+$. 
\begin{enumerate}[\rm (i)]
\item If $a < 0 < b$, then the problem \eqref{eq-Kolm-ysquared} is observable through $\omega$ in time $T$ for any $T>0$.
\item If $a > 0$ there exists $\Tc \geq \frac{a^2}{2}$ such that 
\begin{itemize}
\item if $T > \Tc$ then \eqref{eq-Kolm-ysquared} is observable through $\omega$,
\item if $T < \Tc$ then \eqref{eq-Kolm-ysquared} is not observable through $\omega$.
\end{itemize}
\end{enumerate}
\end{theorem}

The model studied in \cite{Beauchard14} also includes the equation
$$
\partial_t u + y^\gamma \partial_{x} u - \partial_{yy} y = 0
$$
with $\gamma = 1$. In that case, it is proved that the problem is observable through any open set $\omega$, for any $T > 0$, generalizing the previous study \cite{Beauchard_Zuazua_09} where the sets of observation were horizontal strips. Theorem \ref{th-KB} corresponds to the case $\gamma = 2$. The case $\gamma = 3$ is studied in \cite{beauchard_henry}. It is proved that if $0<a<b< \ell_+$ then the problem is not observable through $\T \times (a,b)$ in any time $T>0$.\\

The fact that the observation domain $\omega$ is a horizontal strip of $\O$ may seem quite restrictive. However,
the recent study \cite{koenig18} shows that it is a quasi-necessary condition for \eqref{eq-Kolm-ysquared} to be observable.

\begin{theorem}[\cite{koenig18}] \label{th-AK1}
Let $\omega = \omega_x \times I$, where $\omega_x$ is a strict open set of $\T$.
Then \eqref{eq-Kolm-ysquared} is not observable through $\omega$ in any time $T>0$.
\end{theorem}

Furthermore, it is shown that a minimal time is needed for the system to be possibly observable for most of observation sets $\omega$.

\begin{theorem}[\cite{koenig18}] \label{th-AK2}
Let $\omega$ be an open subset of $\T \times I$. Suppose that there exists $\tilde x \in \T$ and $a > 0$ such that
$$\left\lbrace (\tilde x ,y), \ y \in (-a,a) \right\rbrace \cap \overline{\omega} = \emptyset.$$
Then system \eqref{eq-Kolm-ysquared} is not observable through $\omega$ in any time $T< \frac{a^2}{2}$.
\end{theorem}

In the present paper, we investigate the observability properties of 
\eqref{eq-Kolmogorov} with a more general coefficient $q(y)^2$, when the domain of observation is the boundary 
\[
\Gamma = \partial \O = \T \times \left\lbrace -\ell_-, \ell_+ \right\rbrace.
\]
We could similarly consider observation through an open subset $\o$ given by horizontal strips of $\O$. Our main result is the following:

\begin{theorem} \label{th-main}
We set 
$$
\Tmin =  \frac 1 {q'(0)} \min\left( \int_0^{\ell_+} q(s) \, ds , \int_{-\ell_-}^0 \vert q(s) \vert\, ds \right),
$$
and 
$$
\Tmax =  \frac 1 {q'(0)} \max\left( \int_0^{\ell_+} q(s) \, ds , \int_{-\ell_-}^0 \vert q(s) \vert\, ds \right).
$$
There exists $\Tc \in [\Tmin,\Tmax]$ such that
\begin{enumerate} [\rm (i)]
\item if $T > \Tc$, the problem \eqref{eq-Kolmogorov} is observable through $\Gamma$, \label{item-th-i}
\item if $T < \Tc$, the problem \eqref{eq-Kolmogorov} is not observable through $\Gamma$. \label{item-th-ii}
\end{enumerate}
\end{theorem}

In particular, in any configuration for which $\Tmax=\Tmin$, we obtain the critical time needed
for observability of equation \eqref{eq-Kolmogorov} to hold. This is in particular the case for symmetric configurations:

\begin{theorem} \label{th-cas-symetrique}
Suppose $\ell_- = \ell_+$ and $q$ is odd. Let 
$$
\Tc = \frac{1}{q'(0)} \int_0^{\ell_+} q(s) \, ds.
$$
Then
\begin{enumerate} [\rm (i)]
\item if $T > \Tc$ the problem \eqref{eq-Kolmogorov} is observable through $\Gamma$,
\item if $T < \Tc$ the problem \eqref{eq-Kolmogorov} is not observable through $\Gamma$.
\end{enumerate}
\end{theorem}

Note that in the case $q(y) = y$, the critical time is $\Tc = \frac{\ell_+^2}{2}$. This is the analog for the observation from the boundary of the time $\frac {a^2}2$ which appears in Theorems \ref{th-KB} and \ref{th-AK2}. Theorem \ref{th-cas-symetrique} is, 
up to our knowledge, the first result giving the precise value of the critical time for the observation 
of a two-dimensional Kolmogorov-type equation.

\begin{remark}
By a classical duality argument, Theorem \ref{th-main} is equivalent to controlability properties for the adjoint equation, with 
a boundary Dirichlet control acting on $\Gamma$. We refer to \cite{tucsnak} for details on this equivalence.
\end{remark}

\subsection*{Outline of the paper} The article is organized as follows. After this introduction, we give in Section \ref{sec-strategy} the main ideas for the proof of Theorem \ref{th-main}. The details are then given in the following two sections. In Section \ref{sec-spectral} we discuss the well-posedness of the problem \eqref{eq-Kolmogorov} and we prove some spectral properties for the non-selfadjoint Schr\"odinger operator $K_n = -\partial_{yy} + inq(y)^2$ which naturally appears in the analysis. We prove Agmon-type estimates for the first eigenfunction, which gives the negative result for $T < \Tmin$, and we estimate the decay of the corresponding semigroup. Finally, in Section \ref{sec-Carleman}, we prove a Carleman estimate and deduce an observability estimate in arbitrarily small time which depends on the frequency $n$ with respect to $x$. Together with the decay properties of $e^{-tK_n}$, this will give the observabililty of \eqref{eq-Kolmogorov} for $T > \Tmax$.

\section{Strategy of the proof} \label{sec-strategy}

In this section we describe the strategy for the proof of Theorem \ref{th-main}. We only give the mains ideas, and the details will be postponed to the following two sections.

\subsection{Well-posedness and Fourier transform of the Kolmorgorov equation}

Before discussing the properties of the solutions of \eqref{eq-Kolmogorov}, we check that this problem is well posed.

\begin{proposition} \label{prop-well-posedness}
Let $\uo \in L^2(\O)$. Then there exists a unique 
\[
u \in C^0 \big( [0,T],L^2(\O) \big) \cap C^0 \big( ]0,T],H^2(\O) \cap H^1_0(\O) \big) \cap C^1 \big( ]0,T], L^2(\O) \big)
\]
which satisfies \eqref{eq-Kolmogorov} with $u(0) = \uo$.
\end{proposition}

Notice in particular that the equation is regularizing, so we do not have to impose any regularity on the initial condition to get a solution in the strong sense.\\

Many argument in our analysis, including the proof of Proposition \ref{prop-well-posedness}, will be based on a Fourier transform. All along the paper, the Fourier coefficients are taken with respect to the variable $x \in \T$. Given $u \in L^2(\O)$, we denote by $u_n \in \ell^2(\Z,L^2(I))$ the sequence of Fourier coefficients of $u$:
\[
u(x,y) = \sum_{n\in \mathbb{Z}} u_n(y) e^{i n x}, \quad u_n (y) = \frac 1 {2\pi} \int_\T e^{-inx} u(x,y) \, dx.
\]
The same applies if $u$ (and then the $u_n$, $n \in \Z$) are also functions of the time $t$.\\

For $n \in \Z$ we consider the problem
\begin{equation} \label{eq-Kolmogorov-n}
\begin{cases}
\partial_t u_n - \partial_{yy} u_n + inq(y)^2 u_n = 0, & \text{on } ]0,T[ \times I,\\
u_n(t,-\ell_-) = u_n(t,\ell_+) = 0, & \text{for } t \in ]0,T[,\\
u_n(0)  \in L^2(I).
\end{cases}
\end{equation}
Then the Fourier coefficients of a solution of \eqref{eq-Kolmogorov} are given by the solutions of \eqref{eq-Kolmogorov-n}. 

\begin{proposition} \label{prop-well-posedness-n}
Let $u$ be a solution of \eqref{eq-Kolmogorov} and let $u_n$, $n \in \Z$, be the corresponding Fourier coefficients. Then for all $n \in \Z$ we have
\[
u_n \in C^0 \big( [0,T],L^2(I) \big) \cap C^0 \big( ]0,T],H^2(I) \cap H^1_0(I) \big) \cap C^1 \big( ]0,T], L^2(I) \big),
\]
and $u_n$ is the unique solution of \eqref{eq-Kolmogorov-n} with $u_n(0) = u_{\mathsf{o},n}$, where $u_{\mathsf{o},n}$ is the $n$-th Fourier coefficient of $\uo = u(0)$.
\end{proposition}

An important property of the problem \eqref{eq-Kolmogorov-n} is the following exponential time decay.

\begin{proposition} \label{prop-decay-expKn}
Let
$$
\gamma < \frac{q'(0)}{\sqrt{2}}.
$$
There exists $C>0$ such that for $n \in \Z$, a solution $u_n$ of \eqref{eq-Kolmogorov-n} and $\th_1,\th_2 \in [0,T]$ with $\th_1 \leq \th_2$, one has
$$
\nr{u_n(\th_2)}_{L^2(I)}^2 \leq C \exp \big(- 2 \gamma\,\sqrt{\abs n}(\th_2-\th_1)\big) \nr{u_n(\th_1)}_{L^2(I)}^2.
$$
\end{proposition}

The proofs of Propositions \ref{prop-well-posedness} and \ref{prop-well-posedness-n} will be given in Section \ref{sec-wellposedness-proof}. Proposition \ref{prop-decay-expKn} will be discussed in Section \ref{sec-Henry}.

\subsection{Positive result: upper bound for the critical time}

We begin the proof of Theorem \ref{th-main} with the first statement and prove observability for \eqref{eq-Kolmogorov} when $T > \Tmax$.\\

With the trace theorems, the regularity of the solution ensures that the right-hand side of \eqref{eq-obs-Gamma} makes sense, even if it could be equal to $+\infty$ if the initial condition is not regular enough. In fact, we are going to prove the following stronger result for observability (note that with $\tau_1$ chosen positive, the right-hand side of \eqref{eq_prop-obs-tau12} is finite).

\begin{proposition} \label{prop-obs-tau12}
Let $T > \Tmax$ and $\tau_1 \in ]0,T-\Tmax[$. Let $\tau_2 \in ]\tau_1,T]$. Then there exists $C > 0$ such that for any solution $u$ of \eqref{eq-Kolmogorov} we have 
\begin{equation} \label{eq_prop-obs-tau12}
\nr{u(T)}_{L^2(\O)}^2 \leq C \int_{\tau_1}^{\tau_2} \nr{\partial_\nu u(t)}_{L^2(\partial \O)}^2 \, dt.
\end{equation}
\end{proposition}

Obviously, Proposition \ref{prop-obs-tau12} implies \eqref{eq-obs-Gamma}. The fact that we observe during an arbitrarily small time $\tau_2 - \tau_1$ might seem contradictory with the minimal time condition. It is not the case, since only the state at time $T>T_{max}$ is controled by the observation on the
time interval $[\tau_1,\tau_2]$. As we will see below, the dissipation effect of the Kolmogorov equation plays a key role in obtaining \eqref{eq_prop-obs-tau12}. Roughly speaking, we have to wait long enough for the dissipation to fully play is role, and
inequality \eqref{eq_prop-obs-tau12} to be true.\\

By Proposition \ref{prop-well-posedness-n} and the Parseval identity, Proposition \ref{prop-obs-tau12} is equivalent to an observability estimate for \eqref{eq-Kolmogorov-n} uniform with respect to the Fourier parameter $n$. In other words, it is equivalent to prove the following result.

\begin{proposition} \label{prop-observability-n}
Let $T$, $\tau_1$ and $\tau_2$ be as in Proposition \ref{prop-obs-tau12}. There exists $C>0$ such that for any $n \in \mathbb{Z}$ and any solution $u_n$ of \eqref{eq-Kolmogorov-n} one has
\begin{equation} \label{eq-obs-n}
\Vert u_n(T) \Vert_{L^2(I)}^2 \leq C  \int_{\tau_1}^{\tau_2} \big(\vert \partial_y u_n(t,-\ell_-)\vert^2 + \vert \partial_y u_n(t,\ell_+) \vert^2 \big) \, dt.
\end{equation}
\end{proposition}

Note that it is sufficient to prove \eqref{eq-obs-n} for $n \in \N$. The case $n \in \Z$ then follows by complex conjugation of \eqref{eq-Kolmogorov-n}.\\

The difficulty in Proposition \ref{prop-observability-n} is the uniformity with respect to the parameter $n$. For $n$ fixed, it is already known that the one-dimensional heat equation with a complex-valued potential is observable through the boundary in any positive time:

\begin{proposition} \label{prop-obs-fix-n}
Let $T>0$ and $n \in \mathbb{N}$. Let $\tau_1,\tau_2 \in ]0,T]$ with $\tau_1 < \tau_2$. There exists $C_n > 0$ such that for any solution $u_n$ of \eqref{eq-Kolmogorov-n} we have 
\begin{equation} \label{equ_obsineq_fixedn}
\nr{u_n(T)}_{L^2(I)}^2  \leq C_n \int_{\tau_1}^{\tau_2} \big( \abs {\partial_y u_n(t,-\ell_-)}^2 + \abs{\partial_y u_n(t,\ell_+)}^2 \big) \, dt.
\end{equation}
\end{proposition}
 
A proof of Proposition \ref{prop-obs-fix-n} will be given in Section \ref{sec-obs-fixed-n}. With this result, it is now enough to prove Proposition \ref{prop-observability-n} for $n$ large. To do so, we first obtain a precise estimate
of the constant $C_n$ in the asymptotic $n$ large.

\begin{proposition} \label{prop-obs-large-n}
Let $\tau_1,\tau_2 \in ]0,T]$ with $\tau_1 < \tau_2$ and 
$$
\kappa >  \max \left( \frac 1 {\sqrt 2} \int_0^{\ell_+} q(s) \, ds ,  \frac 1 {\sqrt 2} \int_{-\ell_-}^{0} \vert q(s)\vert  \, ds \right) = \frac {q'(0)}{\sqrt 2} \Tmax.
$$
There exist $n_0 \in \N$ and $C > 0$ such that for $n \geq n_0$ and a solution $u_n$ of \eqref{eq-Kolmogorov-n} one has
$$
 \nr{u_n(\tau_2)}_{L^2(I)}^2 \leq C \exp(2 \kappa \sqrt{n} )  {\int_{\tau_1}^{\tau_2} \big( \vert \partial_y u_n(t,-\ell_-)\vert^2 + \vert \partial_y u_n(t,\ell_+) \vert^2 \big) \,  dt}.
$$
\end{proposition}

The proof of this proposition is based on carefully constructed Carleman estimates, in the spirit of \cite{Beauchard_Darde_Ervedoza}. We refer to Section \ref{sec-observability-fin} for the details.\\

The observability estimate of Proposition \ref{prop-obs-large-n} is valid for any non-trivial interval of time, but it is not uniform with respect to $n$. As said above, the dissipation effect has to be taken into account here. More precisely, the second ingredient for the proof of Proposition \ref{prop-observability-n} is the estimate given by Proposition \ref{prop-decay-expKn}, which precisely counterbalances the loss observed in Proposition \ref{prop-obs-large-n} if we wait long enough.

\begin{proof}[Proof of Proposition \ref{prop-observability-n}, assuming Propositions \ref{prop-decay-expKn}, \ref{prop-obs-fix-n} and \ref{prop-obs-large-n}]
We choose $\delta \in ]0,1[$ so small that 
$$
(1 +\delta) \Tmax < (1-\delta)^2 (T - \tau_1).
$$
Then we set
$$
\kappa =  (1+\delta)\, \frac{q'(0)}{\sqrt{2}}\, \Tmax,
\quad \gamma = (1-\delta) \frac{q'(0)}{\sqrt{2}}.
$$
Proposition \ref{prop-decay-expKn} applied with $\th_2 = T$ and 
\[
\th_1 = \min\big( \tau_1 + \delta (T-\tau_1) , \tau_2 \big)
\]
gives a constant $C_1 > 0$ such that for all $n\in \N$ and $u_n$ solution of \eqref{eq-Kolmogorov-n} we have 
$$
\Vert u_n (T) \Vert_{L^2(I)}^2 \leq C_1 \, \exp \big( -2\gamma \sqrt n (1-\delta) (T-\tau_1) \big) \Vert u_n(\th_1) \Vert_{L^2(I)}^2.
$$
By Propositions \ref{prop-obs-fix-n} and \ref{prop-obs-large-n}, there exists $C_2 > 0$ such that for all $n\in \N$ and $u_n$ solution of \eqref{eq-Kolmogorov-n} we have 
 $$
\Vert u(\th_1) \Vert_{L^2(I)}^2 \leq C_2 \exp(2\kappa \sqrt{n}) {\int_{\tau_1}^{\theta_1}
\big( \vert \partial_y u(t,-\ell_-)\vert^2 + \vert \partial_y u(t,\ell_+) \vert^2 \big) ds }.
$$
Since 
$$
\kappa - \gamma (1-\delta)(T-\tau_1) = \frac{q'(0)}{\sqrt{2}} \left( (1+\delta) \Tmax - (1-\delta)^2 (T-\tau_1) \right)< 0,
$$
these two inequalities give
$$
\Vert u_n(T) \Vert_{L^2(I)}^2 \leq C_1C_2 {\int_{\tau_1}^{\tau_2} \big( \vert \partial_y u(t,-\ell_-)\vert^2 + \vert \partial_y u(t,\ell_+) \vert^2 \big) \, dt },
$$
and the proposition is proved.
\end{proof}

We recall that Proposition \ref{prop-observability-n} implies Proposition \ref{prop-obs-tau12} and hence the first statement of Theorem \ref{th-main}. Thus, it is enough to prove Propositions \ref{prop-decay-expKn}, \ref{prop-obs-fix-n} and \ref{prop-obs-large-n} to get the observability of \eqref{eq-Kolmogorov} through $\G$ for $T > \Tmax$. These proofs are postponed to Sections \ref{sec-spectral} and \ref{sec-Carleman}.

\subsection{Negative result: lower bound for the critical time}

In this paragraph we discuss the second statement of Theorem \ref{th-main} about the non-observability of \eqref{eq-Kolmogorov} if $T < \Tmin$. The proof relies on the construction of a particular family of solutions of \eqref{eq-Kolmogorov} for which the estimate \eqref{eq-obs-Gamma} cannot hold if $T < \Tmin$. In Section \ref{sec-spectral}, we will prove the following result.

\begin {proposition}\label{prop-ln-psin}
For all $n \in \N$, there exist $\l_n \in \C$ and $\psi_n \in H^2(I) \cap H_0^1(I)$ such that $\Vert \psi_n \Vert_{L^2(I)} = 1$,
\begin{equation} \label{eq-lambda-n}
\l_n = \sqrt n q'(0) e^{\frac {i\pi}4} + \littleo{n}{+\infty}(\sqrt n),
\end{equation}
and
\[
\big( -\partial_{yy} + inq(y)^2 \big)\psi_n = \l_n \psi_n.
\]
Moreover, for any $\e > 0$ there exists $C > 0$ such that, for all $n \in \N$,
\begin{equation} \label{eq-Agmon-psi-n}
\vert \psi_n'(-\ell_-) \vert^2 + \vert \psi_n'(\ell_+) \vert^2 \leq 
C n\, \exp\left( -\sqrt{2n}(1-\varepsilon) q'(0) \Tmin \right).
\end{equation}
\end{proposition}

With this proposition we now prove that we cannot have observability through $\G$ in time $T < \Tmin$.

\begin{proof}[Proof of Theorem \ref{th-main}.\eqref{item-th-ii}, assuming Proposition \ref{prop-ln-psin}]
Assume that \eqref{eq-obs-Gamma} holds. For $m \in \N$, $t \in [0,T]$, $x \in \T$ and $y \in \bar I$ we set
$$
u_m(t,x,y) = e^{-\lambda_m t}  e^{i m x} \psi_m(y),
$$
where $\l_m$ and $\psi_m$ are given by Proposition \ref{prop-ln-psin}. This defines a solution $u_m$ of \eqref{eq-Kolmogorov}. Then \eqref{eq-obs-Gamma} gives
$$
2 \Re(\l_m) \leq C \big( e^{2T \Re(\l_m)}-1\big)  \left(  \abs{\psi_m'(-\ell_-)}^2 +  \abs{\psi_m'(\ell_+)}^2 \right).
$$
Let $\varepsilon > 0$. By Proposition \ref{prop-ln-psin} there exists $C_1 > 0$ such that 
$$
(\sqrt{2} q'(0) + o(1))\sqrt{m} \leq C_1 \, m \exp\left( \sqrt{2m} q'(0) \left[ T  - (1-\varepsilon)  \Tmin + o(1) \right] \right).
$$
This implies
$$
T  \geq (1-\varepsilon) \Tmin.
$$
Since this holds for any $\e > 0$, this implies that $T \geq \Tmin$, and the conclusion follows.
\end{proof}

\section{Spectral properties of the Kolmogorov equation} \label{sec-spectral}

In this section we prove Propositions \ref{prop-well-posedness}, \ref{prop-well-posedness-n}, \ref{prop-decay-expKn} and \ref{prop-ln-psin}.

\subsection{Well-posedness and Fourier transform of the Kolmogorov equation} \label{sec-wellposedness-proof}

We begin with the well-posedness of the problems \eqref{eq-Kolmogorov} and \eqref{eq-Kolmogorov-n} for all $n \in \Z$. We also show that if $u$ is a solution of \eqref{eq-Kolmogorov} then its Fourier coefficients $u_n$, $n\in\Z$, are solutions of \eqref{eq-Kolmogorov-n}.\\

We set 
\[
H^1_{0,y}(\O) = \set{u \in L^2(\O) \st \partial_y u \in L^2(\O),\, u(x ,\ell_\pm) = 0 \text{ for almost all } x \in \T }.
\]
By the Poincar\'e inequality, this is a Hilbert space for the norm defined by
\[
\nr{u}_{H^1_{0,y}(\O)}^2 = \nr{\partial_y u}_{L^2(\O)}^2.
\]
We consider on $L^2(\O)$ the operator $K$ defined by 
\[
Ku = -\partial_{yy} u + q(y)^2 \partial_x u
\]
on the domain
\[
\Dom(K) = \set{u \in H^1_{0,y}(\O) \st K u \in L^2(\O)},
\]
where $Ku$ is understood in the sense of distributions. Similarly, for $n \in \Z$ we consider on $L^2(I)$ the operator
\begin{equation} \label{def-Kn}
K_n = -\partial_{yy} + i n q(y)^2,
\end{equation}
defined on the domain (independent of $n$)
\begin{equation} \label{dom-Kn}
\Dom(K_n) = H^2(I) \cap H_0^1(I).
\end{equation}

We notice that $K_0$ is just the usual Dirichlet Laplacian on $I$. In particular, it is selfadjoint and non-negative. However, the operators $K$ and $K_n$ for $n \neq 0$ are not symmetric. We will show that they are at least accretive. For $K$ this means that
\[
\forall u \in \Dom(K), \quad \Re \innp{K u}{u}_{L^2(\O)} \geq 0.
\]
In fact, they are even maximal accretive. This means in particular that any $z \in \C$ with $\Re(z) < 0$ belongs to the resolvent set of $K$.

\begin{proposition} 
\begin{enumerate} [\rm (i)]
\item The operator $K$ is maximal accretive on $L^2(\O)$.
\item For all $n \in \Z$, the operator $K_n$ is maximal accretive on $L^2(I)$.
\item Let $u \in \Dom(K)$ and let $(u_n)_{n \in \Z}$ be the Fourier coefficients of $u$. Then $u_n$ belongs to $\Dom(K_n)$ for all $n \in \Z$ and the Fourier coefficients of $Ku$ are the $K_n u_n$, $n \in \Z$.
\end{enumerate}
\label{prop-K-Kn}
\end{proposition}

\begin{proof}
\stepp We begin with the second statement. It is easy to see that for $n \in \Z$ and $u \in \Dom(K_n)$ we have 
\begin{equation} \label{eq-Knun-un}
\Re \innp{K_n u}{u} = \nr{u'}_{L^2(I)}^2 \geq 0,
\end{equation}
which means that $K_n$ is accretive. Then $K_n$ is an accretive and bounded perturbation of the selfadjoint operator $K_0$, so it is maximal accretive.

\stepp Now let $u \in \Dom(K)$ and $v = Ku \in L^2(\O)$. We denote by $(u_n)_{n \in \Z}, (v_n)_{n \in \Z} \in \ell^2(\Z,L^2(I))$ the sequences of Fourier coefficients of $u$ and $v$, respectively. Let $n \in \Z$, $\vf_n \in C_0^\infty(I)$ and $\vf : (x,y) \mapsto e^{inx} \vf_n(y)$. By the Parseval identity we have
\begin{multline*}
\innp{u_n}{-\vf_n'' - in q(y)^2 \vf_n}_{L^2(I)}
= \frac 1 {2\pi} \innp{u}{-\partial_{yy} \vf - q(y)^2 \partial_x \vf}_{L^2(\O)}\\
= \frac 1 {2\pi} \innp{v}{\vf}_{L^2(\O)} = \innp{v_n}{\vf_n}_{L^2(I)}.
\end{multline*}
This implies that $u_n'' \in L^2(I)$ (hence $u_n \in H^2(I)$) and 
\[
-u_n'' + in q(y)^2 u_n = v_n.
\]
On the other hand, it is clear from the definition of $u_n$ that $u_n(-\ell_-) = u_n(\ell_+) = 0$, so $u_n \in \Dom(K_n)$. Then we can write $K_n u_n = v_n$. This gives the last statement of the proposition.

\stepp As above we can see that the Fourier coefficients of $\partial_y u$ are the $u_n'$, $n \in \Z$. Then, by \eqref{eq-Knun-un} and the Parseval identity we get
\begin{equation} \label{eq-K-accretive}
\nr{\partial_y u }_{L^2(\O)}^2 = 2\pi \sum_{n \in \Z} \nr{u_n'}_{L^2(I)}^2 = 2\pi \Re \sum_{n \in \Z} \innp{v_n}{u_n} = \Re \innp{v}{u} = \Re \innp{Ku}{u}.
\end{equation}

\stepp We check that the operator $K$ is closed. Let $(u_m)_{m \in \N}$ be a sequence in $\Dom(K)$ such that $u_m \to u$ and $K u_m \to v$ in $L^2(\O)$, for some $u,v \in L^2(\O)$. In the sense of distributions we have
\begin{equation} \label{eq-Ku-v}
-\partial_{yy} u + q(y)^2 \partial_x u = \lim_{m \to +\infty} \left( -\partial_{yy} u_m + q(y)^2 \partial_x u_m  \right) = v  \in L^2(\O).
\end{equation}
For $m,p \in \N$ we have $u_m - u_p \in \Dom(K)$, so by \eqref{eq-K-accretive} we have 
\begin{align*}
\nr{u_m - u_p}_{H^1_{0,y}(\O)}^2
= \Re \innp{K (u_m-u_p)}{u_m-u_p}_{L^2(\O)}^2
\limt m {+\infty} 0.
\end{align*}
This implies that the sequence $(u_m)_{m \in \N}$ has a limit in $H^1_{0,y}(\O)$, which is necessarily $u$. By the trace theorem, we see that $u_m$ also goes to $u$ in $L^2(\partial\O)$, so $u$ vanishes on $\partial \O$. Finally, we have proved that $u$ belongs to $\Dom(K)$ and, by \eqref{eq-Ku-v}, $Ku = v$. This proves that $K$ is closed.

\stepp By \eqref{eq-K-accretive} the operator $K$ is accretive on $L^2(\O)$. Then, for $u \in \Dom(K)$ we have 
\begin{equation} \label{eq-norm-K+1}
\begin{aligned}
\nr{(K+1)u}_{L^2(\O)}^2  \geq \nr{Ku}_{L^2(\O)}^2 + \nr{u}_{L^2(\O)}^2,
\end{aligned}
\end{equation}
so $(K+1)$ is injective with closed range. Now let $v \in L^2(\O)$ be such that 
\[
\forall u \in \Dom(K), \quad \innp{(K+1)u}{v} = 0.
\]
Then, in the sense of distributions we have 
\[
-\partial_{yy} v - q(y)^2 \partial_x v + v = 0.
\]
As above we can check that the operator $\tilde K = -\partial_{yy} - q(y)^2 \partial_x$, defined on the domain
\[
\Dom(\tilde K) = \set{u \in H^1_{0,y}(\O) \st \tilde K u \in L^2(\O)},
\]
is accretive. This implies that $v = 0$ (in fact, $\tilde K$ is the adjoint of $K$). Thus $\Ran(K+1)^\bot = \{0\}$ and $(K+1)$ is invertible. By \eqref{eq-norm-K+1}, its inverse is bounded. This proves that $-1$ belongs to the resolvent set of $K$, and hence $K$ is maximal accretive.
\end{proof}

By the Lummer-Philipps Theorem (see for instance \cite{engel}), the operator $(-K)$ generates a contractions semigroup $(e^{-tK})_{t \geq 0}$ on $L^2(\O)$. Given $\uo \in \Dom(K)$, the function $u : t \mapsto e^{-tK} u$ belongs to $C^0(\R_+,\Dom(K)) \cap C^1(\R_+,L^2(\O))$. This gives a strong solution of \eqref{eq-Kolmogorov}. More generally, for $\uo \in L^2(\O)$, the function $t \mapsto e^{-tK} \uo$ belongs to $C^0(\R_+,L^2(\O))$. This gives a weak solution of \eqref{eq-Kolmogorov}. The same applies on $L^2(I)$ to $K_n$ and \eqref{eq-Kolmogorov-n} for any $n \in \Z$. In particular, if $u_n$ is a solution of \eqref{eq-Kolmogorov-n} then for all $t_1,t_2 \in [0,T]$ such that $t_1 \leq t_2$ we have 
\begin{equation} \label{eq-semigroup-contraction}
\nr{u_n(t_2)}_{L^2(I)}^2 \leq \nr{u_n(t_1)}_{L^2(I)}^2.
\end{equation}

Now, we show that the solutions of \eqref{eq-Kolmogorov-n} for $n \in \Z$ give the Fourier coefficients of a solution of \eqref{eq-Kolmogorov}

\begin{proposition} \label{prop-Fourier-propagateurs}
Let $\uo \in L^2(\O)$. For $t \geq 0$ we set $u(t) = e^{-tK} \uo \in L^2(\O)$. We denote by $(\uon)_{n \in \Z}$ and $(u_n(t))_{n \in \Z}$ the Fourier coefficients of $\uo$ and $u(t)$, $t \geq 0$, respectively. Then for all $n \in \Z$ and $t \geq 0$ we have 
\[
u_n(t) = e^{-t K_n} \uon.
\]
\end{proposition}

\begin{proof}
First assume that $\uo \in \Dom(K)$. Let $n \in \Z$ and $t \geq 0$. By differentiation under the integral sign and Proposition \ref{prop-K-Kn} we have in $L^2(I)$, for $h > 0$,
\begin{multline*}
\frac {u_n(t+h) - u_n(t)} h = \frac 1 {2\pi} \int_\T e^{-inx} \frac {u(t+h,x) - u(t,x)} h \, dx\\
\limt h 0 - \frac 1 {2\pi} \int_\T e^{-inx} K u(t,x) \, dx = -K_n u_n(t).
\end{multline*}
The conclusion follows in this case.

In general, since $\Dom(K)$ is dense in $L^2(\O)$, we can consider a sequence $(u_{\mathsf{o}}^m)_{m \in \N}$ in $\Dom(K)$ which converges to $\uo$ in $L^2(\O)$. For $m \in \N$ we denote by $u_{\mathsf{o},n}^m$, $n \in \Z$, the Fourier coefficients of $u_{\mathsf{o}}^m$. Then $u_{\mathsf{o},n}^m$ goes to $u_{\mathsf{o},n}$ in $L^2(I)$ for all $n \in \Z$. Then for $n \in \Z$ we have by continuity of $e^{-tK}$ and $e^{-tK_n}$ in $L^2(\O)$ and $L^2(I)$ respectively
\begin{align*}
u_n(t) 
& = \frac 1 {2\pi} \int_\T e^{-inx} (e^{-tK} \uo)(x) \,dx\\
& = \lim_{m \to +\infty}  \frac 1 {2\pi} \int_\T e^{-inx} (e^{-tK} u_{\mathsf{o}}^m)(x) \,dx\\
& = \lim_{m \to +\infty} e^{-tK_n} u_{\mathsf{o},n}^m\\
& = e^{-tK_n} \uon.
\end{align*}
The proposition is proved.
\end{proof}

Before we can state Theorem \ref{th-main}, we still have to check that the right-hand side of \eqref{eq-obs-Gamma} makes sense (one would not have this difficulty with observabililty through an open subset of $\O$). 
To do so, we investigate the regularizing effect of equation \eqref{eq-Kolmogorov}, and prove that
even if the initial condition $\uo$ merely belongs to $L^2(\O)$, the solution is smooth enough for the right-hand side
of \eqref{eq-obs-Gamma} to be well defined.
The proof of this result relies on Proposition \ref{prop-decay-expKn}, which will be proved in Section \ref{sec-Henry} below.

\begin{proposition} \label{prop-regularite}
For $\uo \in L^2(\O)$ and $\tau > 0$ we have 
\[
e^{-\tau K} \uo \in H^2(\O) \cap H^1_0(\O) \subset \Dom(K).
\]
\end{proposition}

\begin{proof}
\stepp For $t > 0$ we set $u(t) = e^{-tK} \uo$. We denote by $\uon$ and $u_n(t)$, $n \in \Z$, the Fourier coefficients of $\uo$ and $u(t)$, respectively. For $n \in \Z$ and $t > 0$ we have $u_n(t) = e^{-t K_n} \uon$ by Proposition \ref{prop-Fourier-propagateurs}.

\stepp By Proposition \ref{prop-decay-expKn} there exists $c > 0$ such that for $\tau > 0$ and $k \in \N$ we have
\begin{equation} \label{eq-estim-nk-un}
n^k \nr{u_n(\tau)}_{L^2(I)} \leq \frac c {\tau^{2k}} \nr{\uon}_{L^2(I)}.
\end{equation}
This implies in particular that $u(\tau) \in C^\infty(\T,L^2(I))$.

\stepp Let $\tau > 0$. Assume that $\uo \in C_0^\infty(\O) \subset \Dom(K^2)$. Then $u \in C^1([0,\tau], \Dom(K))$ and $u_n \in C^1([0,\tau], \Dom(K_n))$ for all $n \in \Z$. Let $n \in \Z$. Since $(\partial_t + K_n) u_n = 0$ we have 
\begin{multline*}
0 = \Re \int_\tau^{2\tau} (t-\tau) \innp{(\partial_t + K_n) u_n(t)}{\partial_t u_n(t)}_{L^2(I)} \, dt\\
 = \int_\tau^{2\tau} (t-\tau) \nr{\partial_t u_n(t)}_{L^2(I)}^2 \, dt + \int_\tau^{2\tau} (t-\tau) \, \Re \innp{K_n u_n(t)}{\partial_t u_n(t)}_{L^2(I)} \, dt. 
\end{multline*}
Since 
\begin{eqnarray*}
\lefteqn{\Re \innp{K_n u_n(t)}{\partial_t u_n(t)}}\\
&& = \Re \innp{-\partial_{yy} u_n(t)}{\partial_t u_n(t)} + \Re \innp{inq(y)^2 u_n(t)}{\partial_t u_n(t)}\\
&& \geq \frac 12 \frac d {dt} \innp{-\partial_{yy} u_n(t)}{u_n(t)} - n \nr{q}_\infty^2 \nr{u_n(t)} \nr{\partial_t u_n(t)}\\
&& \geq \frac 12 \frac d {dt} \Re \innp{K_n u_n(t)}{u_n(t)} - \frac {n^2 \nr{q}_\infty^4 \nr{u_n(t)}^2}2  - \frac {\nr{\partial_t u_n(t)}^2} 2,
\end{eqnarray*}
with \eqref{eq-semigroup-contraction} this gives 
\begin{multline*}
\int_{\tau}^{2\tau} (t-\tau) \frac {\nr{\partial_t u_n(t)}^2} 2 \, dt\\
\leq - \frac 12  \Re \int_{\tau}^{2\tau} (t-\tau) \frac d {dt} \innp{K_n u_n(t)}{u_n(t)} \, dt +  \frac {n^2 \tau^2 \nr{q}_\infty^4 \nr{u_n(\tau)}^2} 4.
\end{multline*}
An integration by parts gives 
\begin{eqnarray*}
\lefteqn{- \Re \int_\tau^{2\tau} (t-\tau) \frac d {dt} \innp{K_n u_n(t)}{u_n(t)} \, dt}\\
&& = - \tau \Re \innp{K_n u_n(2\tau)}{u_n(2\tau)} + \int_\tau^{2\tau} \Re \innp{K_n u_n(t)}{u_n(t)} \, dt\\
&& \leq - \frac 12 \int_\tau^{2\tau} \frac d {dt} \nr{u_n(t)}^2 \,dt\\
&& \leq \frac {\nr{u_n(\tau)}^2} 2 - \frac {\nr{u_n(2\tau)}^2} 2.
\end{eqnarray*}
On the other hand, since the function $t \mapsto \partial_t u_n(t)$ is also a solution of \eqref{eq-Kolmogorov}, its norm is non-increasing, so 
\[
\frac 12 \int_\tau^{2\tau} (t-\tau) \nr{\partial_t u(t)}^2 \, dt \geq \frac {\tau^2 \nr{\partial_t u_n(2 \tau)}^2} 4.
\]
Finally, with \eqref{eq-estim-nk-un} we get 
\begin{align*}
\nr{\partial_t u_n(2 \tau)}^2 + \frac {2 \nr{u_n(2\tau)}^2}{\tau^2}
& \leq \frac {2 \nr{u_n(\tau)}^2}{\tau^2} +  n^2 \nr{q}_\infty^4 \nr{u_n(\tau)}^2 \\
& \leq \frac {2 \nr{\uon}^2}{\tau^2} +  \frac {c^2 \nr{q}_\infty^4 \nr{\uon}^2}{\tau^4}.
\end{align*}
Hence, by the Parseval identity, 
\begin{equation} \label{eq-Ku-tau}
\nr{Ku(2\tau)}^2 = \nr{\partial_t u(2\tau)}_{L^2(\O)}^2 \leq  \frac {2 \nr{\uo}^2}{\tau^2} +  \frac {c^2 \nr{q}_\infty^4 \nr{\uo}^2}{\tau^4}.
\end{equation}

\stepp Let $\uo \in L^2(\O)$ and $(\uom)_{m \in \N}$ be a sequence in $C_0^\infty(\O)$ which goes to $\uo$ in $L^2(\O)$. For $\tau > 0$ we set $u(\tau) = e^{-\tau K} \uo$ and $u_m(\tau) = e^{-\tau K} \uom$, $m \in \N$. Let $\d > 0$. $u_m(t)$ converges to $u(t)$ for any $t \geq 0$ and the function $t \mapsto u_m'(t)$ has a uniform limit on $[\d,+\infty[$. This implies that the function $u$ belongs to $C^1(]0,+\infty[,L^2(\O))$. Then, since $-K$ is the generator of the semigroup $e^{-tK}$, $u(t)$ belongs to $\Dom(K)$ for all $t > 0$ and $u'(t) = -K u(t)$.

\stepp Finally, for $\uo \in L^2(\O)$ and $\tau > 0$ we have $(-\partial_{yy} - q(y)^2 \partial_x) u(\tau) \in L^2(\O)$ and $\partial_x u(\tau) \in L^2(\O)$, so $-\partial_{yy}u(\tau) \in L^2(\O)$. Since we also have $\partial_{xx}u(\tau) \in L^2(\O)$, this proves that $u(\tau)$ belongs to $H^2(\O)$. The fact that $u(\tau)$ is also in $H_0^1(\O)$ is a consequence of the fact that it is in $\Dom(K) \subset H^1_{0,y}(\O)$, and the proof is complete.
\end{proof}

\subsection{General spectral properties for non-selfajdoint Schr\"odinger operators}

In the rest of this section, we prove Propositions \ref{prop-decay-expKn} and \ref{prop-ln-psin}. They can both be rewritten in terms of the operator $K_n$ defined by \eqref{def-Kn}-\eqref{dom-Kn}.\\

We have seen in Proposition \ref{prop-K-Kn} that $K_n$ is a maximal accretive operator on $L^2(I)$. In particular, the resolvent set of $K_n$ is not empty. And since $\Dom(K_n)$ is compactly embedded in $L^2(I)$, the resolvent of $K_n$ is compact. This implies that the spectrum of $K_n$ consists of eigenvalues which have finite algebraic multiplicities.\\ 

We have already said that $K_n$ generates a contractions semigroup on $L^2(I)$ (see \eqref{eq-semigroup-contraction}). However, this is not enough for Proposition \ref{prop-decay-expKn}. For $n = 0$, the operator $K_0$ is selfadjoint and the decay of the corresponding semigroup is given by the functional calculus. If we denote by $\l_0$ the first eigenvalue of $K_0$, then $\l_0$ is positive and for all $t \geq 0$ we have
\[
\nr{e^{-t K_0}}_{\Lc(L^2(I))} \leq e^{-t\l_0}.
\]
For $n \neq 0$, the operator $K_n$ is not selfadjoint, and the link between the exponential decay of $e^{-tK_n}$ and the real parts of the eigenvalues of $K_n$ is not that direct.\\

The purpose of the rest of this section is then to give some spectral properties for the non-selfadjoint operator $K_n$. We are interested in the location of the spectrum (and in particular the eigenvalue with the smallest real part), the size of the resolvent $(K_n-z)\inv$ for $z$ outside this spectrum (for a non-selfadjoint operator, the resolvent can have a large norm even for $z$ far from the spectrum) and then an estimate of the propagator $e^{-t K_n}$ for $t \geq 0$.\\

The properties of the operator $K_n$ will be deduced from analogous results for the classical complex harmonic oscillators and the complex Airy operators.\\

With the Agmon estimates (see Paragraph \ref{sec-Agmon} below), we will see that for large $n$ the eigenvectors of $K_n$ associated to ``small'' eigenvalues should be in some sense localized close to 0. And near 0 we have 
\[
in q(y)^2 \sim i n q'(0)^2 y^2.
\]
Thus, it is expected that, at least for a small spectral parameter, the spectral properties of $K_n$ for large $n$ should be close to those of the harmonic oscillator
\begin{equation} \label{def-Hn}
H_n = -\partial_{yy} + i n q'(0)^2 y^2,
\end{equation}
defined on the domain 
\[
\Dom(H_n) = \set{ u \in H^2(\R) \st yu \in H^1(\R), y^2 u \in L^2(\R)}.
\]

It is known (see for instance \cite[\S 14.4]{helffer}) that $H_n$ defines for all $n \in \N^*$ a maximal accretive operator on $L^2(\R)$. Its spectrum consists of a sequence of (geometrically and algebraically) simple eigenvalues, given by 
\[
(2k - 1) \sqrt n q'(0) e^{\frac {i\pi}4} , \quad k \in \N^*,
\]
and for each $k \in \N^*$, a corresponding eigenfunction is given by 
\begin{equation} \label{eq-eigenfunction-Hn}
y \mapsto P_k(e^{\frac {i\pi}8} \a y) e^{- \frac {(\a y)^2}{2\sqrt 2} - \frac {i(\a y)^2}{2\sqrt 2}}, \quad \a = n^{\frac 14} q'(0)^{\frac 12},
\end{equation}
where $P_k$ is a polynomial of degree $k$. In particular,
\[
\inf_{\s \in \Sp(H_n)} \Re(\s) = \frac {\sqrt n q'(0)}{\sqrt 2}.
\]
This is not enough to get a decay estimate for the propagator $e^{-t H_n}$, $t \geq 0$. However, it is also known that for $\gamma < \frac {q'(0)}{2}$ there exists $c > 0$ such that 
\begin{equation} \label{estim-Hn-resolvent}
\sup_{\Re(z) \leq \g \sqrt n} \nr{(H_n-z)\inv}_{\Lc(L^2(\R))} \leq \frac {c}{\sqrt n}
\end{equation}
(in fact we have more precise resolvent estimates \cite{HitrikSjoVio13,KrejcirikSieTatVio15}). Then we deduce (see for instance \cite{engel} for the theory of semigroups) that there exists $C > 0$ such that for all $t \geq 0$ we have 
\begin{equation} \label{estim-Hn-propagator}
\nr{e^{-tH_n}}_{\Lc(L^2(\R))} \leq C e^{-t\gamma \sqrt n}.
\end{equation}
Proposition \ref{prop-decay-expKn} precisely says that we have a similar estimate for the propagator generated by $K_n$, while Proposition \ref{prop-ln-psin} shows that for large $n$ the first eigenvalue of $K_n$ is close to the first eigenvalue of $H_n$. The decay of the corresponding eigenfunction has the same form as in \eqref{eq-eigenfunction-Hn}, but it depends on the values of $q$ on the whole interval $I$, and not only on its behavior in a neighborhood of 0.\\

For the proofs, we will also compare $K_n$ to some complex Airy operators. Near $y_0 \in I \setminus \set 0$, the potential $in q(y)^2$ looks like $inq(y_0)^2 + 2in q(y_0) q'(y_0) (y-y_0)$, with $q(y_0) q'(y_0) \neq 0$. It is then useful to recall the properties of Schr\"odinger operators with linear purely imaginary potentials.\\

Given $\a \in \R \setminus \set 0$, we consider on $L^2(\R)$ the operator 
\begin{equation} \label{def-Aa}
A_\a u = -\partial_{yy} u + i \a y u,
\end{equation}
defined on the domain
\begin{align*}
\Dom(A_\a) = \set{u \in L^2(\R) \st (-u'' + i \a y u) \in L^2(\R)}.
\end{align*}
This complex Airy operator is now well understood, see for instance \cite{Helffer11,KrejcirikSie15} and references therein. We notice that for $\a > 0$ we have 
\[
\Th_\a \inv (A_\a - z)\inv \Th_\a = \frac 1 {\a^{\frac 23}} \big( A_1 - z \a^{-\frac 23} \big) \inv,
\]
where $\Th_a$ is the unitary operator defined on $L^2(\R)$ by 
\[
(\Th_\a u)(y) = \a^{\frac 16} u \big( \a^{\frac 13} y \big).
\]
Moreover, $A_{-\a} = A_\a^*$. Then, from the properties of $A_{1}$ we deduce the following result.

\begin{proposition}
\begin{enumerate}[\rm (i)]
\item The spectrum of $A_\a$ is empty for any $\a \in \R\setminus \set 0$.
\item Let $\gamma \in \R$. Then there exists $c > 0$ such that for all $\a \in \R \setminus \set 0$ we have 
\[
\sup_{\Re(z) \leq \g \abs \a^{\frac 2 3}} \nr{(A_\a -z)\inv} \leq \frac c {\abs \a^{\frac 23}}.
\]
\end{enumerate} \label{prop-Airy-R}
\end{proposition}

To understand the behavior of $K_n$ near the boundary points $\pm \ell_\pm$, we introduce the complex Airy operator on $\R_+$. For $\a \in \R \setminus \set 0$ we consider on $L^2(\R_+)$ the operator defined by 
\[
A_\a^+ u = -\partial_{yy} u + i \a y u,
\]
on the domain 
\[
\Dom(A_\a^+) = \set{u \in L^2(\R_+) \st  (-u'' + i \a y u) \in L^2(\R_+) \text{ and } u(0) = 0}.
\]
To prove the following proposition, we use in $L^2(\R_+)$ the same dilation $\Th_a$ as above and we apply \cite[Lemma 5.1]{Helffer11}. For the properties of the Airy function we refer for instance to \cite{ValleeSoares}.

\begin{proposition} 
\begin{enumerate}[\rm(i)]
\item Let $\a > 0$. The spectrum of $A_\a^+$ consists of a sequence of simple eigenvalues. These eigenvalues are given by 
\[
\l_k^+ = \a^{\frac 23} e^{\frac {i\pi}3} \abs {\m_k}, \quad k \in \N^*,
\]
where $\dots < \mu_k < \dots < \mu_2 < \mu_1 < 0$ are the zeros of the Airy function.
\item Let $\gamma < \frac {\abs {\mu_1}} 2$. There exists $C > 0$ such that for all $\a \in \R \setminus \set 0$ we have
\[
\sup_{\Re(z) \leq \g \abs \a^{\frac 2 3}} \nr{(A_\a^+ -z)\inv} \leq \frac c {\abs \a^{\frac 23}}.
\]
\end{enumerate}\label{prop-Airy-R+}
\end{proposition}

Of course, we have similar properties on $L^2(\R_-)$ for the operator 
\[
A_\a^- : u \mapsto  -\partial_{yy} u + i \a y u,
\]
defined on the domain 
\[
\Dom(A_\a^-) = \set{u \in L^2(\R_-) \st  (-u'' + i \a y u) \in L^2(\R_-) \text{ and } u(0) = 0}.
\]

\subsection{Resolvent estimates}\label{sec-Henry}

In this paragraph, we prove Proposition \ref{prop-decay-expKn} (see Proposition \ref{prop-exp-Kn} below) and the first part of Proposition \ref{prop-ln-psin}, about the eigenvalue $\l_n$ (see Proposition \ref{prop-lambda-n}). The estimate of a corresponding eigenfunction at the boundary will be given in the next paragraph.\\

We prove estimates for the resolvent $(K_n-z)\inv$ when $z$ has real part smaller than $\gamma \sqrt n$, with $\gamma$ as in Proposition \ref{prop-decay-expKn}. More precisely, we estimate the difference between $(K_n-z)\inv$ and the model resolvent $(H_n-z)\inv$, in a suitable sense. By the theory of semigroups, this will give Proposition \ref{prop-decay-expKn}. This will also give the existence of an eigenvalue $\l_n$ which satisfies \eqref{eq-lambda-n}.\\

To compare $(K_n-z)\inv$ and $(H_n-z)\inv$, we follow the ideas of \cite{henry}. Our one-dimensional setting is simpler than the general case considered therein so, for the reader convenience, we provide a complete proof adapted to our problem. Notice also that \eqref{eq-lambda-n} is not contained in the results given in \cite{henry}, where the imaginary parts of the eigenvalues are not an issue.\\

We denote by $\1_I$ the operator which maps $u \in L^2(\R)$ to its restriction on $I$: $\1_I u = u|_{I} \in L^2(I)$. Then $\1_I^*$ maps a function $v \in L^2(I)$ to its extansion by 0 on $\R$.

\begin{proposition}
\begin{enumerate}[\rm(i)]
\item Let $\gamma \in \big]0,\frac {q'(0)}{\sqrt 2}\big[$. There exist $n_0 \in \N^*$ and $c > 0$ such that for $n \geq n_0$ and $z \in \C$ with $\Re(z) \leq \gamma \sqrt n$ we have $z \in \rho(K_n)$ and 
\[
\nr{(K_n-z)\inv}_{\Lc(L^2(I))} \leq \frac {c}{\sqrt n}.
\]
\item We have 
\[
\nr{\1_I^* K_n\inv \1_I - H_n\inv }_{\Lc(L^2(\R))} = \littleo n {+\infty} \left( \frac 1 {\sqrt n} \right).
\]
\end{enumerate} 
\label{prop-spectre-Kn}
\end{proposition}

\begin{proof} 
The proof consists in using localized versions of the resolvents of the complex harmonic operator $H_n$ and of Airy-type operators to construct an approximation $Q_n(z)$ of the resolvent $(K_n-z)\inv$. We first introduce suitable cut-off functions, then we define $Q_n(z)$ and finally we check that it is indeed an approximation
of $(K_n -z)^{-1}$ up to a uniformly bounded operator. The proposition will then follow from estimates on $Q_n(z)$.
For $n \in \N^*$ we set 
\[
\C_n^- = \set{z \in \C \st \Re(z) \leq \g \sqrt n}.
\]

\stepp 
For $n \in \N^*$ and $z \in \C_n^-$ we set 
\[
R_n(z) = \1_I (H_n-z)\inv \1_I^*.
\]
This defines a bounded operator on $L^2(I)$. Our purpose is to prove that $R_n(z)$ gives an approximate inverse of $(K_n-z)$ near 0, in the following sense. We consider 
\[
\rho \in \left]\frac 16,\frac 14 \right[,
\]
a cut-off function $\h \in C_0^\infty(\R,[0,1])$ supported in $I$ and equal to 1 on a neighborhood of 0, and for $n \in \N^*$ and $y \in \bar I$ we set 
\[
\h_n(y) = \h(n^\rho y).
\]
Then we set 
\[
T_n(z) = R_n(z) \chi_n (K_n-z) - \chi_n.
\]
We prove that $T_n(z)$ extends to a bounded operator on $L^2(I)$ and 
\begin{equation} \label{estim-res-Rn}
\nr{T_n(z)}_{\Lc(L^2(I))} \limt n {+\infty} 0,
\end{equation}
where the convergence is uniform with respect to $z \in \C_n^-$.

Let $u \in \Dom(K_n)$. For $n \in \N^*$ we have $\chi_n u \in \Dom(K_n)$ and $\1_I^* \chi_n u \in \Dom(H_n)$. For $y \in \bar I$ we set 
\[
r(y) = q(y)^2 - q'(0)^2 y^2.
\]
Then for $z \in \C_n^-$ we have 
\[
R_n(z) (K_n-z) \chi_n u = \chi_n u + i n R_n(z) r \chi_n u.
\]
This gives 
\begin{equation} \label{eq-K-Rn}
R_n(z) \chi_n (K_n-z) u = \chi_n u + i n R_n(z) r \chi_n - R_n(z) \chi_n'' u + 2 R_n(z) (\chi_n' u)'.
\end{equation}
Since $\abs{r(y) \h_n(y)} \lesssim n^{-3\rho}$, we have by \eqref{estim-Hn-resolvent}
\begin{equation*}
\nr{n   R_n(z) r \h_n}_{\Lc(L^2(I))} \lesssim n^{1 - 3\rho - \frac 12} \limt n {+\infty} 0.
\end{equation*}
We also have 
\[
\nr{R_n(z) \chi_n''}_{\Lc(L^2(I))} \lesssim n^{2\rho - \frac 12} \limt n {+\infty} 0.
\]
For the last term we observe that for $v \in L^2(\R)$ we have
\begin{eqnarray*}
\lefteqn{\nr{\partial_y (H_n^*-\bar z)\inv v}_{L^2(\R)}^2}\\
&& = \Re \innp{(H_n^* - \bar z) (H_n^*-\bar z)\inv v}{(H_n^*-\bar z)\inv v}_{L^2(\R)} + \Re(z) \nr{(H_n^*-\bar z)\inv v}_{L^2(\R)}^2\\
&& \lesssim \frac {\nr{v}_{L^2(\R)}^2}{\sqrt n}.
\end{eqnarray*}
Taking the adjoint gives 
\begin{multline*}
\nr{R_n(z) \partial_y (\chi_n' u)}_{L^2(I)} \leq \nr{(H_n-z)\inv \partial_y (\1_I^* \h_n' u)}_{L^2(\R)}\\
\lesssim n^{-\frac 14} \nr{\1_I^* \chi_n'u}_{L^2(\R)} \lesssim n^{\rho-\frac 14} \nr{u}_{L^2(I)},
\end{multline*}
and \eqref{estim-res-Rn} follows.

\stepp 
Then we consider
\[
\tilde \rho \in \left] \frac {1+2\rho} 6, \frac {1-\rho} 3 \right[.
\]
In particular, $\tilde \rho > \rho$. For $n \in \N^*$ we denote by $\nu_n$ the integer part of $1+ (\ell_+ + \ell_-) n^{\tilde \rho}$, and for $j \in \Ii 0 {\nu_n}$ we set 
\[
a_{j,n} = -\ell_- + j \d_n, \quad \d_n =  \frac {\ell_+ + \ell_-}{\nu_n}.
\]
We also consider $\theta \in C_0^\infty(\R)$ supported in $\big]-\frac 23,\frac 23\big[$, equal to 1 on $\big[-\frac 13,\frac 13\big]$ and such that $\th(-y) = 1-\th(1-y)$ for $y \in [0,1]$. Then for all $y \in \R$ we have 
\[
\sum_{m \in \Z} \theta(y-m) = 1.
\]
For $n \in \N^*$, $j \in \Ii 1 {\nu_n}$ and $y \in \bar I$ we set 
\[
\theta_{j,n}(y) = \theta \left( \frac {y-a_{j,n}}{\d_n} \right) (1-\chi_n)(y).
\]
Let $A_{j,n}$ be defined by
\[
A_{j,n} u = -\partial_{yy} u + in q(a_{j,n})^2 + 2 i n q(a_{j,n}) q'(a_{j,n}) (y-a_{j,n}) u
\]
on the domain 
\[
\Dom(A_{j,n}) = \set{u \in H^2(\R) \st y u \in L^2(\R)}.
\]
With the notation \eqref{def-Aa} we have 
\[
A_{j,n} = \tau_{a_{j,n}} A_{2n(qq')(a_{j,n})} \tau_{-a_{j,n}} + i n q(a_{j,n})^2,
\]
where $\tau_{\pm a_{j,n}}$ is the usual translation operator: $(\tau_{\pm a_{j,n}} u)(y) = u(y \mp a_{j,n})$. Thus $A_{j,n}$ satisfies the properties of Proposition \ref{prop-Airy-R} with $\a = 2n (qq')(a_{j,n})$. We similarly set
\[
A_{0,n} = \tau_{-\ell_-} A_{2n(qq')(-\ell_-)}^+ \tau_{\ell_-}  + i n q(-\ell_-)^2
\]
and 
\[
A_{\nu_n,n} = \tau_{\ell_+} A_{2n(qq')(\ell_+)}^- \tau_{-\ell_+} + i n q(\ell_+)^2.
\]
Notice that $A_{0,n}$ is an operator on $L^2(-\ell_-,+\infty)$ and $A_{\nu_n,n}$ is an operator on $L^2(-\infty,\ell_+)$. They satisfy the same properties as the model operators (see Proposition \ref{prop-Airy-R+}). 

For $j \in \Ii 1 {\nu_n-1}$ we set $\1_j = \1_I$. We also denote by $\1_0$ the operator which maps $u \in L^2(-\ell_-,+\infty)$ to its restriction on $I$, and by $\1_{\nu_n}$ the operator which maps $u \in L^2(-\infty,\ell_+)$ to its restriction to $I$. For $n \in \N^*$, $z \in \C_n^-$ and $j \in \Ii 0 {\nu_n}$ we set 
\[
R_{j,n}(z) = \1_j \big( A_{j,n}-z \big)\inv \1_j^*
\]
and 
\[
T_{j,n}(z) = R_{j,n}(z) \theta_{j,n} (K_n-z) - \theta_{j,n}.
\]

We proceed as above. For $j \in \Ii 0 {\nu_n}$ we have 
\begin{equation*}
R_{j,n}(z) \theta_{j,n} (K_n-z)u = \theta_{j,n} u+ i n R_{j,n}(z) \theta_{j,n} r_{j,n} u - R_{j,n}(z) \th_{j,n}'' u + 2 R_{j,n}(z) \partial_y (\th_{j,n}'u)
\end{equation*}
where
\[
r_{j,n}(y) = q^2(y) - q^2(a_{j,n}) - 2\, (y-a_{j,n}) q'(a_{j,n}) q(a_{j,n}).
\]
Let $n \in \N^*$ and $j \in \Ii 0 {\nu_n}$ be such that $\theta_{j,n} \neq 0$. Then $\abs{a_{j,n}} \gtrsim n^{-\rho}$, $2n q'(a_{j,n}) q(a_{j,n})\gtrsim n^{1-\rho}$ and hence, for $z \in \C_n^-$ (in particular $\Re(z) \leq \g n^{\frac 12} \ll n^{\frac 23(1-\rho)}$), Proposition \ref{prop-Airy-R+} gives
\begin{equation} \label{estim-res-Ajn}
\nr{R_{j,n}(z)}_{\Lc(L^2(I))} \lesssim n^{-\frac 23(1-\rho)}.
\end{equation}
Then, as above we have  $\abs{\tilde r_{j,n}(y) \theta_{j,n}(y)} \lesssim n^{-2\tilde \rho}$ so 
\[
n \nr{ R_{j,n}(z) \tilde r_{j,n} \theta_{j,n}}_{\Lc(L^2(I))} \lesssim n^{1-2\tilde \rho - \frac 23(1-\rho)} \limt n {+\infty} 0.
\]
Moreover,
\begin{eqnarray*}
\nr{R_{j,n}(z) \th_{j,n}'' u }_{L^2(I)} &\lesssim& n^{2\tilde \rho - \frac 23 (1-\rho)} \nr{u}_{L^2(I)},\\
\nr{R_{j,n}(z) (\th'_{j,n} u)'}_{\Lc(L^2(I))} &\lesssim& n^{\tilde \rho - \frac 13(1-\rho)} \nr{u}_{L^2(I)}.
\end{eqnarray*}
All these estimates being uniform with respect to $j \in \Ii 0 {\nu_n}$, we finally get
\begin{equation} \label{estim-res-Rjn}
\sup_{z \in \C_n^-} \sup_{0\leq j \leq \nu_n} \nr{T_{j,n}(z)}_{\Lc(L^2(I))} \limt n {+\infty} 0.
\end{equation}

\stepp
For $u \in L^2(I)$ we write
\[
u = \h_n u + \sum_{j=0}^{\nu_n} \th_{j,n} u,
\]
We want to sum \eqref{estim-res-Rn} and the estimates \eqref{estim-res-Rjn}, for $j \in \Ii 0 {\nu_n}$, to get an approximate inverse for $(K_n-z)$. We have seen that each contribution goes to 0, but the number of terms grows with $n$.

Let $\tilde \th \in C_0^\infty(\R,[0,1])$ be equal to 1 on $\big[-\frac 23,\frac 23\big]$ and supported in $]-1,1[$. Then for $n \in \N^*$, $j \in \Ii 1 {\nu_n}$ and $y \in \bar I$ we set 
\[
\tilde \th_{j,n}(y) = \tilde \th \left( \frac {y-a_{j,n}}{\d_n} \right),
\]
and then 
\begin{equation} \label{eq_def_Qn}
Q_n(z) = R_n(z) \h_n + \sum_{j=0}^{\nu_n} \tilde \theta_{j,n}  R_{j,n}(z) \theta_{j,n}.
\end{equation}
For $u \in \Dom(K_n)$ and $z \in \C_n^-$ we have $\th_{j,n} u = \th_{j,n} \tilde \th_{j,n} u$ and
\[
\theta_{j,n} (K_n-z) (1- \tilde \theta_{j,n}) u = 0,
\]
so
\[
\nr{Q_n(z) (K_n-z) u - u }_{L^2(I)} \leq \nr{T_n(z) u}_{L^2(I)} + \nr{\sum_{j=0}^{\nu_n} \tilde \th_{j,n} T_{j,n}(z) \tilde \th_{j,n} u}_{L^2(I)}.
\]
Moreover $\tilde \th_{j,n} \tilde \th_{k,n} = 0$ whenever $\abs{j-k} \geq 2$, so by almost orthogonaly (twice) we can write
\begin{align*}
\nr{\sum_{j=0}^{\nu_n} \tilde \th_{j,n} T_{j,n}(z) \tilde \th_{j,n} u}^2_{L^2(I)} 
& \lesssim \sum_{j=0}^{\nu_n} \nr{\tilde \th_{j,n} T_{j,n}(z) \tilde \th_{j,n} u}^2_{L^2(I)}\\
& \lesssim \sup_{0\leq j \leq \nu_n} \nr{T_{j,n}(z)}_{\Lc(L^2(I))}^2 \sum_{j=0}^{\nu_n} \nr{\tilde \th_{j,n} u}^2_{L^2(I)}\\
& \lesssim \sup_{0\leq j \leq \nu_n} \nr{T_{j,n}(z)}_{\Lc(L^2(I))}^2 \nr{u}^2_{L^2(I)}.
\end{align*}
This proves 
\begin{equation} \label{eq-Qn-Kn-z}
\sup_{z \in \C_n^-} \sup_{\substack{u \in \Dom(K_n)\\ \nr{u}_{L^2(I)} = 1}} \nr{Q_n(z) (K_n-z) u - u}_{L^2(I)} \limt n {+\infty} 0.
\end{equation}
Thus for $n$ large enough the operator $K_n$ has no eigenvalue and hence no spectrum in $\C_n^-$. Moreover for $z \in \C_n^-$ we have
\begin{equation} \label{eq-Qngoodapprox}
(K_n-z)\inv = B_n(z) Q_n(z),
\end{equation}
where
\[
B_n(z) = \big(1 + \big( Q_n(z) (K_n-z) - 1 \big) \big)\inv 
\]
is bounded on $L^2(I)$ uniformly in $z \in \C_n^-$ and $n$ large enough.

\stepp Let $u \in L^2(I)$ and $z \in \mathbb{C}_n^-$. By \eqref{estim-res-Ajn}, and using again the almost orthogonality, we obtain 
\[
\nr{B_n(z) \sum_{j=0}^{\nu_n} \tilde \theta_{j,n} R_{j,n}(z) \theta_{j,n}  u}^2_{L^2(I)} \lesssim \frac {\nr{u}^2_{L^2(I)}}{n ^{\frac 43(1-\rho)}}, 
\]
so 
\begin{equation} \label{eq-res-Kn-Rn}
\nr{(K_n-z)\inv - B_n(z) R_n(z) \h_n }_{\Lc(L^2(I))} \lesssim \frac 1 {n^{\frac 23(1-\rho)}}.
\end{equation}
With \eqref{estim-Hn-resolvent}, this gives the first statement of the proposition.

\stepp We now consider the case $z = 0$ to prove the second part of the proposition. By \eqref{eq-Qn-Kn-z} we have
\[
\nr{B_n(0) - 1}_{\Lc(L^2(I))} = \nr{ \big( 1 + (Q_n(0) K_n-1) \big)\inv - 1}_{\Lc(L^2(I))} \limt n {+\infty} 0,
\]
so \eqref{eq-res-Kn-Rn} and \eqref{estim-Hn-resolvent} give 
\begin{equation} \label{eq-Kn-Rn-0}
\nr{K_n\inv - \1_I H_n\inv \1_I^* \h_n}_{\Lc(L^2(I))} = \littleo n {+\infty} \left( \frac 1 {\sqrt n} \right).
\end{equation}
On $\supp(1-\chi_n)$ we have $\abs y \gtrsim n^{-\rho}$, so for $u \in L^2(I)$ we can write 
\begin{align*}
\nr{(1-\chi_n) \1_I (H_n^*)\inv \1_I^* u}^2_{L^2(I)}
& \lesssim n^{2\rho} \nr{y (H_n^*)\inv \1_I^*u}^2_{L^2(\R)} \\
& \lesssim n^{2\rho - 1} \abs{\Im \innp{(H_n^* (H_n^*)\inv \1_I^* u}{(H_n^*)\inv \1_I^* u}_{L^2(\R)}} \\
& \lesssim \frac {\nr u^2_{L^2(I)}}{n^{\frac 32 - 2\rho}}.
\end{align*}
Taking the adjoint gives
\[
\nr{\1_I H_n\inv \1_I^* (1-\chi_n)}_{\Lc(L^2(I))} = \littleo n {+\infty} \left( \frac 1 {\sqrt n} \right).
\]
With \eqref{eq-Kn-Rn-0}, the proof is complete.
\end{proof}

Now we are in position to prove Proposition \ref{prop-decay-expKn}. It is a direct consequence of the following result.

\begin{proposition} \label{prop-exp-Kn}
Let $\gamma < \frac {q'(0)}{\sqrt 2}$. There exist $n_0 \in \N^*$ and $C > 0$ such that for $n \geq n_0$ and $t \geq 0$ we have 
\[
\nr{e^{-tK_n}}_{\Lc(L^2(I))} \leq C e^{-t\gamma \sqrt n}.
\]
\end{proposition}

\begin{proof}
Let $n_0 \in \N^*$ and $c > 0$ by given by Proposition \ref{prop-spectre-Kn}. For $n \geq n_0$ we set $\tilde K_n = -K_n + \g \sqrt n$. Then for $n \geq n_0$ and $z \in \C$ with $\Re(z) \geq  0$ we have $z \in \rho(\tilde K_n)$ and 
\[
\big\| (\tilde K_n - z)\inv \big\|_{\Lc(L^2(I))} \leq \frac {c}{\sqrt n}.
\]
Moreover for $t \geq 0$ we have 
\[
\big\|e^{t \tilde K_n} \big\| \leq e^{t\g \sqrt n}.
\]
Then we apply \cite[Th. V.1.11 p. 302]{engel} to the operator $\tilde K_n$. With the notation used in the proof therein, we have $\o_0 \leq \g \sqrt n$, $M\leq \frac c {\sqrt n}$ and $L= 2\pi$. We obtain that the semigroup $(e^{t\tilde K_n})_{t \geq 0}$ is bounded uniformly in $t \geq 0$ and $n \geq n_0$, so there exists $C > 0$ such that for all $n \geq n_0$ and $t \geq 0$ we have 
\[
\nr{e^{-tK_n}} = e^{-t\g \sqrt n} \big \| e^{t\tilde K_n} \big \| \leq C e^{-t\g \sqrt n}.
\]
We also refer to \cite{HelfferSjo10} to get bounds on a semigroup from bounds on the resolvent of the corresponding generator.
\end{proof}

Now we turn to the proof of \eqref{eq-lambda-n}. A more general version of the following result is given in \cite[\S IV.3.5]{kato}.

\begin{proposition} \label{prop-perturbation-vp-kato}
Let $T$ be a closed operator on a Hilbert space $\mathscr H$. Let $\l \in \C$. Assume that $\l$ is an isolated eigenvalue of $T$. Let $(B_m)_{m \in \N}$ be a sequence of bounded operators on $\mathscr H$ such that $\nr{B_m}_{\Lc(\Hc)} \to 0$ as $m \to +\infty$. For $m \in \N$ we set $T_m = T+B_m$. Let $\e > 0$. Then for $m$ large enough the operator $T_m$ has an eigenvalue $\l_m$ such that $\abs{\l_m-\l} \leq \e$.
\end{proposition}

\begin{proof}
We set $\mathcal C = \set {\z \in \C, \abs {\z - \l}  = \e}$. Without loss of generality, we can assume that $\e> 0$ is so small that $\l$ is the only point of $\Sp(T)$ in the disk $D(\l,2\e)$. We set $M = \sup_{\z \in \mathcal C} \nr{(T-\z)\inv}$. Since 
\[
T_m - \z = (T-\z) \big( 1 + (T-\z)\inv B_m \big),
\]
we see that $\mathcal C \cap \Sp(T_m) = \emptyset$ as soon as $M \nr{B_m} < \frac 1 2$. Moreover, in this case, we have for $\z \in \mathcal C$,
\[
\nr{(T_m-\z)\inv} \leq 2 M.
\]
We set 
\[
P = \frac 1 {2i\pi} \int_{\mathcal C} (T-\z)\inv \, d\z.
\]
We similarly define $P_m$ by replacing $T$ by $T_m$. Then we have by the resolvent identity
\[
\nr{P_m - P} = \nr{\frac 1 {2i\pi} \int_{\mathcal C} (T-z)\inv B_m (T_m-\z)\inv \, d\z} \leq 2\e M^2 \Vert B_m \Vert.
\]
Thus for $m$ small enough we have $\nr{P_m - P} < 1$. By \cite[\S I.4.6]{kato} this implies that 
\[
\dim (\Ran (P_m)) = \dim (\Ran (P)) \in \N^*.
\]
This proves that $T_m$ has an eigenvalue $\l_m$ such that $\abs{\l - \l_m} < \e$.
\end{proof}

\begin{proposition} \label{prop-lambda-n}
For $n \in \N^*$ large enough there exists an eigenvalue $\l_n$ of $K_n$ such that 
\[
\abs{\l_n -e^{i\frac \pi 4} q'(0) \sqrt n } = \littleo n {+\infty} \left(\sqrt n \right). 
\]
\end{proposition}

\begin{proof}
We consider on $L^2(\R)$ the unitary operator $\Th_n$ which maps $u$ to 
\[
\Th_n u : x \mapsto n^{\frac 18} u \left( n^{\frac 14} x \right).
\]
Then we have $\Th_n\inv H_n \Th_n = \sqrt n H_1$. By Proposition \ref{prop-spectre-Kn},
\[
\nr{\sqrt n \Th_n\inv \1_I^* K_n\inv \1_I \Th_n - H_1\inv}_{\Lc(L^2(\R))} \limt n {+\infty} 0.
\]
We set $\l = e^{i\frac \pi 4} q'(0)$. Then $\m = \l\inv$ is an eigenvalue of $H_1\inv$. By Proposition \ref{prop-perturbation-vp-kato}, there exists an eigenvalue $\m_n$ of $\sqrt n \Th_n\inv \1_I^* K_n \inv \1_I \Th_n$ such that $\m_n$ goes to $\m$ as $n$ goes to $+\infty$. Then $n^{-\frac 12} \m_n$ is an eigenvalue of $\1_I^* K_n\inv \1_I$, and hence an eigenvalue of $K_n\inv$. We conclude the proof by setting $\l_n = \sqrt n \m_n\inv$.
\end{proof}

\subsection{Agmon estimates} \label{sec-Agmon}

To conclude the proof of Proposition \ref{prop-ln-psin}, it remains to prove the estimate \eqref{eq-Agmon-psi-n} for an eigenfunction $\psi_n$ of $K_n$ corresponding to the eigenvalue $\l_n$.

This estimate is given by an Agmon estimate. The Agmon estimates measure how the eigenfunctions corresponding to the smallest eigenvalues of a Schr\"odinger operator concentrate near the minimum of the potential. Exponential decay of eigenfunctions and precise Agmon estimates are classical results for real-valued potentials (see for instance \cite{Agmon85,helffer1}). We refer to \cite{KrejcirikRayRoySie17} for Agmon estimates for a general non-selfadjoint Laplacian.

Here, it is expected that for large $n$ an eigenfunction corresponding to the first eigenvalue $\l_n$ of $K_n$ will concentrates near 0, where the potential $q^2$ reaches its minimum. In particular, such an eigenfunction will be small at the boundary, so it is indeed a good candidate to break an observability estimate like \eqref{equ_obsineq_fixedn} when $T < \Tmin$.\\

\begin{proposition} \label{prop-Agmon}
Let $E > 0$ and $\e \in ]0,1[$. For $n \in \N$ and $y \in \bar I$ we set 
\begin{equation} \label{def-W}
W_{n,\e}(y) = \frac {1-\e} {\sqrt 2} \abs{\int_0^y \sqrt{\big( nq(s)^2 - \sqrt{n}(E + \e) \big)_+} \, ds},
\end{equation}
where for $\s \in \R$ we write $\s_+$ for $\max(0,\s)$. There exists $C > 0$ such that for $n \in \N$, $u \in \Dom(K_n)$ and $\l \in \C$ with 
\begin{equation} \label{hyp-lambda}
\abs{\Re(\l)} + \abs{\Im(\l)}  \leq E \sqrt n,
\end{equation}
we have 
\begin{multline*}
\nr{e^{W_{n,\e}} u'}_{L^2(I)}^2 + \sqrt n \nr{e^{W_{n,\e}} u}_{L^2(I)}^2
\leq C \sqrt n \nr{u}_{L^2(I)}^2 +\frac {C}{\sqrt n} \nr{e^{W_{n,\e}} (K_n-\l)u}_{L^2(I)}^2.
\end{multline*}
\end{proposition}

This result is proved with more generality in \cite{KrejcirikRayRoySie17}. For the reader convenience we recall a proof in our 1-dimensional setting.

\begin{proof}
We denote by $Q_n$ the quadratic form corresponding to $K_n$. It is defined for $f,g \in H_0^1(I)$ by
\[
Q_n(f,g) = \int_I f' \bar g'  + i n \int_I q^2  f \bar g.
\]

\stepp Let $u \in \Dom(K_n)$. For $\zeta \in W^{1,\infty}(\bar I,\R)$, we have
\begin{align*}
\innp{u'}{(\zeta^2 u)'}_{L^2(I)} = \innp{\z u'}{2\z'u + \z u'}_{L^2(I)} = \innp{(\zeta u)' - \zeta' u}{(\zeta u)' + \zeta' u}_{L^2(I)},
\end{align*}
so
\begin{equation*}
\Re \innp{u'}{(\zeta^2 u)'}_{L^2(I)} = \nr{(\zeta u)'}^2_{L^2(I)} - \nr{\zeta' u}^2_{L^2(I)}.
\end{equation*}

\stepp 
Let $W \in W^{1,\infty}(\bar I,\R)$. Applied with $\zeta = e^{W}$, this equality gives
\begin{equation*}
\Re \big (Q_n (u,e^{2W}u) \big) = \Re \innp{u'}{(e^{2W} u)'}_{L^2(I)} = \nr{(e^{W}u)'}^2_{L^2(I)} - \nr{W' e^W u}^2_{L^2(I)}.
\end{equation*}
On the other hand, a direct computation shows that
\[
\Im \big (Q_n (u,e^{2W}u) \big) = \Im \innp{u'}{2W' e^{2W} u}_{L^2(I)} + n \nr{q e^W u}^2_{L^2(I)}.
\]
Let $\a \in ]0,1[$. Since 
\begin{align*}
\abs{\Im \innp{u'}{2W' e^{2W} u}_{L^2(I)}}
& = 2 \abs{\Im \innp{(e^W u)'}{W' e^{W} u}_{L^2(I)}} \\
& \leq \a \nr{(e^W u)'}^2_{L^2(I)} + \a\inv\nr{W' e^W u}^2_{L^2(I)},
\end{align*}
we have 
\[
\Im \big (Q_n (u,e^{2W}u) \big) \geq n \nr{q e^W u}^2_{L^2(I)} - \a \nr{(e^W u)'}^2_{L^2(I)} - \a\inv\nr{W' e^W u}^2_{L^2(I)},
\]
and hence 
\begin{multline*} 
\Re \big (Q_n (u,e^{2W}u) \big) + \Im \big (Q_n (u,e^{2W}u) \big)\\
\geq  (1-\a) \nr{(e^{W}u)'}^2_{L^2(I)} + \int_{I} \left( nq^2 - (1+\a\inv) W'^2  \right) |e^W u|^2 .
\end{multline*}
Finally,
\begin{align*}
\nr{(e^{W}u)'}^2_{L^2(I)}
& \geq \nr{e^W u'}^2_{L^2(I)} + \nr{W' e^W u}^2_{L^2(I)} - 2\nr{e^W u'}_{L^2(I)}\nr{W' e^W u}_{L^2(I)}\\
& \geq \frac 12 \nr{e^W u'}^2_{L^2(I)} - \nr{W' e^W u}^2_{L^2(I)},
\end{align*}
so if we set $\beta = 2 + \a\inv - \a$ and $\e_1 = \frac {1-\a} 2$, we get
\begin{multline} \label{minor-Re-Im-Q}
\Re \big (Q_n (u,e^{2W}u) \big) + \Im \big (Q_n (u,e^{2W}u) \big)\\
\geq  \e_1 \nr{e^{W}u'}^2_{L^2(I)} + \int_{I} \left( nq^2 - \beta W'^2  \right) |e^W u|^2 .
\end{multline}

\stepp On the other hand, for $\l \in \C$ we have 
$$
Q_n (u,e^{2W}u)  = \l \nr{e^W u}^2_{L^2(I)} +  \innp{(K_n-\l)u}{e^{2W}u}_{L^2(I)}.
$$
We take the real and imaginary parts of this equality. With \eqref{minor-Re-Im-Q} this gives
\begin{multline} \label{ineq-Agmon}
\e_1 \nr{e^{W}u'}^2_{L^2(I)} + \int_{I} \left( nq^2 - \beta W'^2 - \Re(\l) - \Im(\l) \right) |e^W u|^2 \\
\leq 2 \nr{e^W (K_n-\l)u}_{L^2(I)} \nr{e^W u}_{L^2(I)}.
\end{multline}

\stepp Now assume that \eqref{hyp-lambda} holds. Let $\d_n^\pm \in ]0, \ell_\pm]$ be such that 
\[
[-\d_n^-,\d_n^+] = \set{y \in \bar I \st nq(y)^2 \leq \sqrt{n}(E+\e)}.
\] 
Let $W_{n,\e}$ be given by \eqref{def-W}. We choose $\alpha \in ]0,1[$ in such a way that
\[
\beta = \frac {2}{(1-\e)^2}.
\]
On $[-\d_n^-,\d_n^+]$, $W_{n,\e}$ and hence $W'_{n,\e}$ vanish, so
$$ 
\beta W_{n,\e}'(y)^2 + \Re(\l)  + \Im(\l) - nq(y)^2   \leq E \sqrt n,
$$
while on $I \setminus [-\d_n^-,\d_n^+]$ we have 
\[
\b W_{n,\e}'(y)^2 = nq(y)^2 - \sqrt{n}(E+\e),
\]
so 
\[
n q(y)^2 - \beta W_{n,\e}'(y)^2 - \Re(\l) - \Im(\l) \geq \e \sqrt n.
\]
Then, by \eqref{ineq-Agmon}, 
\begin{eqnarray*}
\lefteqn{\e_1 \nr{e^{W_{n,\e}} u'}^2_{L^2(I)} + \e \sqrt n \int_{I \setminus [-\d_n^-,\d_n^+]} \abs{e^{W_{n,\e}} u}^2 }\\
&& \leq 2 \nr{e^{W_{n,\e}} (K_n-\l)u}_{L^2(I)} \nr{e^{W_{n,\e}} u}_{L^2(I)} + E \sqrt n  \int_{-\d_n^-}^{\d_n^+}  \abs{u}^2\\
&& \leq \frac {\e \sqrt n}2 \nr{e^{W_{n,\e}} u}^2_{L^2(I)} + \frac {2}{\e \sqrt n}  \nr{e^{W_{n,\e}} (K_n-\l)u}^2_{L^2(I)} +  E \sqrt n \int_{-\d_n^-}^{\d_n^+}  \abs{u}^2,
\end{eqnarray*}
and finally,
\begin{multline*}
\e_1 \nr{e^{W_{n,\e}} u'}^2_{L^2(I)} + \e \sqrt n \nr{e^{W_{n,\e}} u}^2_{L^2(I)}\\
\leq \frac {\e \sqrt n}{2} \nr{e^{W_{n,\e}} u}^2_{L^2(I)} + \frac {2}{\e \sqrt n}  \nr{e^{W_{n,\e}} (K_n-\l)u}^2_{L^2(I)} +  (E+\e)  \sqrt n \int_{-\d_n^-}^{\d_n^+} \abs{u}^2.
\end{multline*}
The proposition is proved.
\end{proof}

For $\e \in ]0,1]$ and $y \in I$ we set 
\[
\kappa_\e(y) = \frac {(1-\e)} {\sqrt 2}  \int_0^y q(s) \, ds.
\]
We first check that the estimate of Proposition \ref{prop-Agmon} still holds with $W_{n,\e}$ replaced by $\sqrt n \kappa_\e$.

\begin{proposition} \label{prop-W}
Let $E > 0$ and $\e \in ]0,1]$. There exists $C_\e > 0$ such that for $n \in \N$ and $y \in I$ we have 
\begin{equation} \label{eq-minor-W}
W_{n,\e/2}(y) \geq \sqrt n \kappa_\e (y) - C_\e.
\end{equation}
\end{proposition}

\begin{proof}
It is enough to prove the inequality for $n$ large. Let $\a \geq 1$ to be fixed large enough later. For $n$ large enough we consider $\eta_n^\pm \in ]0,\ell_\pm]$ such that 
\[
q(\pm \eta_n^\pm)^2 = \frac {\a}{\sqrt n} \left(E+ \frac \e 2 \right).
\]
We have 
\[
\eta_n^\pm = \mathop{\Oc}_{n \to +\infty} \big( n^{-\frac 14} \big), 
\]
and hence
\begin{equation*}
\sqrt n \kappa_\e(\pm \eta_n^\pm)  =   \mathop{\Oc}_{n \to +\infty} (1).
\end{equation*}
In particular, for $n$ large enough the inequality \eqref{eq-minor-W} holds for $y \in [-\eta_n^-,\eta_n^+]$ if $C_\e$ is chosen large enough, since then the right-hand side is negative. On the other hand, for $y \geq \eta_n^+$ we have 
\begin{align*}
\int_{\eta_n^+}^{y} \sqrt{nq(s)^2 - \left(E+\frac \e 2 \right)\sqrt n} \, ds \geq \sqrt{1-\a \inv } \sqrt n \int_{\eta_n^+}^y q(s) \, ds.
\end{align*}
Then 
\[
W_{n,\e/2}(y) 
\geq \frac {1-\frac \e 2}{1 - \e} \sqrt{1-\a \inv } \sqrt n \kappa_\e(y) + \mathop{\Oc}_{n \to +\infty} (1).
\]
For $\a$ large enough this gives \eqref{eq-minor-W}. We proceed similarly for $y \leq -\eta_n^-$.
\end{proof}

Combining Propositions \ref{prop-Agmon} and \ref{prop-W} we obtain the following version of the Agmon estimates:

\begin{proposition} \label{prop-psi-L2}
Let $E > 0$ and $\e \in ]0,1]$. There exists $C > 0$ such that for $n \in \N$, $u \in \Dom(K_n)$ and $\l \in \C$ with $\abs{\Re(\l)} + \abs{\Im(\l)} \leq E \sqrt n$ we have  
\begin{equation*}
\big\|e^{\sqrt n \kappa_\e} u'\big\|_{L^2(I)}^2 + \sqrt n \big\|e^{\sqrt n \kappa_\e} u\big\|_{L^2(I)}^2
\leq C \sqrt n \nr{u}_{L^2(I)}^2 + \frac {C}{\sqrt n} \big\|e^{\sqrt n \k_\e} (K_n-\l)u\big\|_{L^2(I)}^2.
\end{equation*}
\end{proposition}

\begin{proof}
If we denote by $\tilde C > 0$ the constant given by Proposition \ref{prop-Agmon}, then by Proposition \ref{prop-W} we obtain the estimate of Proposition \ref{prop-psi-L2} with $C = e^{C_\e} \tilde C$.
\end{proof}

From Proposition \ref{prop-psi-L2} we deduce the pointwise estimate \eqref{eq-Agmon-psi-n}. 

\begin{proposition} \label{prop-psi-Linf}
Let $E > 0$ and $\e \in ]0,1[$. There exists $C > 0$ such that for $n \in \N$, an eigenvalue $\m_n$ of $K_n$ with $\Re(\m_n) + \Im(\m_n) \leq E \sqrt n$ and $\psi_n \in \ker(K_n-\m_n)$, we have  
\[
\big\|e^{\sqrt n \kappa_\e} \psi_n'\big\|_{L^\infty(I)}^2 \leq C n \nr{\psi_n}_{L^2(I)}^2.
\]
\end{proposition}

\begin{proof}
By Proposition \ref{prop-psi-L2} we have 
\begin{equation} \label{eq-psi-psi'}
\big\|e^{\sqrt n \kappa_\e} \psi_n\big\|_{L^2(I)}^2 \lesssim \nr{\psi_n}_{L^2(I)}^2, \quad 
\big\|e^{\sqrt n \kappa_\e} \psi_n'\big\|_{L^2(I)}^2 \lesssim C \sqrt n \nr{\psi_n}_{L^2(I)}^2.
\end{equation}

\stepp  We prove 
\begin{equation} \label{eq-psi''}
\big\|e^{\sqrt n \kappa_\e} \psi_n''\big\|_{L^2(I)}^2 \lesssim n^{\frac 32} \nr{\psi_n}_{L^2(I)}^2.
\end{equation}
We have $\psi_n'' = inq^2 \psi_n - \mu_n \psi_n$. With \eqref{eq-psi-psi'} we get
\[
\big\| e^{\sqrt n \kappa_\e} \mu_n \psi_n \big\|_{L^2(I)}^2 \lesssim \abs{\mu_n}^2 \nr{\psi_n}_{L^2(I)}^2 \lesssim n \nr{\psi_n}_{L^2(I)}^2.
\]
For the other term we have by an integration by parts
\begin{align*}
\sqrt 2 (1-\e) \big\| e^{\sqrt n \kappa_\e} n q^2 \psi_n \big\|_{L^2(I)}^2
& = \int_I  2 \sqrt n \k_\e' e^{2 \sqrt n \kappa_\e} n^{\frac 32} q^3 \abs{\psi_n}^2 \, dy\\
& = -\int_I  e^{2 \sqrt n \kappa_\e} n^{\frac 32} \big(3 q^2 q' \abs{\psi_n}^2 + 2 q^3 \Re(\overline{\psi_n} \psi_n') \big) \, dy.
\end{align*}
On the one hand we have 
\[
\abs{\int_I  e^{2 \sqrt n \kappa_\e} n^{\frac 32} 3 q^2 q' \abs{\psi_n}^2 \, dy} \lesssim n^{\frac 32} \big\| e^{\sqrt n \kappa_\e} \psi_n \big\|_{L^2(I)}^2 \lesssim n^{\frac 32} \nr{\psi_n}_{L^2(I)}^2.
\]
On the other hand,
\begin{multline*}
\abs{\int_I  e^{2 \sqrt n \kappa_\e} n^{\frac 32} 2 q^3 \Re(\overline{\psi_n} \psi_n') \big) \, dy} \leq 2 \big\|e^{\sqrt n \kappa_\e} n q^2 \psi_n\big\|_{L^2(I)} \big\| q e^{\sqrt n \kappa_\e}  \sqrt n \psi_n'\big\|_{L^2(I)}\\
\leq (1-\e) \big\|e^{\sqrt n \kappa_\e} n q^2 \psi_n\big\|_{L^2(I)}^2 + \frac {\nr{q}_\infty^2}{1-\e} \big\|e^{\sqrt n \kappa_\e}  \sqrt n \psi_n'\big\|_{L^2(I)}^2.
\end{multline*}
This gives \eqref{eq-psi''}.

\stepp 
Since $\psi_n'$ vanishes on $I$ (we could also use the general trace Theorem) we have by \eqref{eq-psi-psi'} and \eqref{eq-psi''}
\begin{align*}
\big\|e^{\sqrt n \kappa_\e} \psi_n'\big\|_{L^\infty(I)}^2 
& \leq 2 \big\|e^{\sqrt n \kappa_\e} \psi_n'\big\|_{L^2(I)} \big\|\big(e^{\sqrt n \kappa_\e} \psi_n'\big)'\big\|_{L^2(I)}\\
& \lesssim n^{\frac 14} \nr{\psi_n}_{L^2(I)} \left( \big\| \sqrt n \kappa_\e' e^{\sqrt n \kappa_\e} \psi_n'\big\|_{L^2(I)} + \big\|e^{\sqrt n \kappa_\e} \psi_n''\big\|_{L^2(I)} \right)\\
& \lesssim n \nr{\psi_n}_{L^2(I)}^2.
\end{align*} 
This completes the proof.
\end{proof}

Notice that \eqref{eq-psi''} is better that the naive estimate obtained from \eqref{eq-psi-psi'} and the expression of $\psi_n''$. In fact we do not have to be optimal here, since the power of $n$ in the right-hand side of \eqref{eq-Agmon-psi-n} is not important for the proof of the second part of Theorem \ref{th-main}.

\section{The Observability estimate in small time} \label{sec-Carleman}

In this section we prove Propositions \ref{prop-obs-fix-n} (see Paragraph \ref{sec-obs-fixed-n}) and \ref{prop-obs-large-n} (see Paragraph \ref{sec-observability-fin}). The proofs rely on some Carleman estimates and the construction of a suitable weight function.

In this section we will not use an index $n$ for a solution $u$ of \eqref{eq-Kolmogorov-n}. No confusion will be possible since we will never consider a solution of the initial $x$-dependent problem \eqref{eq-Kolmogorov}. Moreover, we use an index for the partial derivatives, so $u_t$ stands for $\partial_t u$, $u_{yy}$ for $\partial_{yy} u$, etc.

\subsection{A generic Carleman estimate}

We begin our analysis with a generic Carleman estimate. In the following statement, $\vf$ is a Carleman weight function. It will be applied to $w = e^{-\vf} u$, where $u$ is a solution of a problem of the form \eqref{eq-Kolmogorov-n}, possibly with a source term (see \eqref{eq-Kolmogorov-f} below). We also impose that $w$ vanishes at initial and final times.\\

\begin{proposition} \label{prop-Carleman}
Let $n \in \N$, $\tau_1,\tau_2 > 0$ with $\tau_1<\tau_2$, $a,b \in \R$ with $a < b$, and $g \in L^2 (]\tau_1,\tau_2[\times ]a,b[)$. Let $\phi \in C^4(]\tau_1,\tau_2[ \times [a,b],\R_+)$. We consider $w \in C^0([\tau_1,\tau_2],H^2(a,b)) \cap C^1([\tau_1,\tau_2],L^2(a,b))$ such that
 \begin{equation} \label{eq-w}
w_t - w_{yy} + inq(y)^2 w +  \phi_t w - 2  \phi_y w_y -  \phi_y^2 w -  \phi_{yy} w = g.
\end{equation}
We assume that $w$ also satisfies the Dirichlet boundary condition
\begin{equation} \label{eq-w-Dir}
\forall t \in ]\tau_1,\tau_2[, \quad w(t,a) = w(t,b) = 0,
\end{equation}
and the initial and final conditions
\begin{equation} \label{eq-w-0T}
\forall y \in ]a,b[, \quad w(\tau_1,y) = w(\tau_2,y) = 0, \quad w_y(\tau_1,y) = w_y(\tau_2,y) = 0.
\end{equation}
Then we have 
\begin{equation*} 
\int_{\tau_1}^{\tau_2} \int_a^b \big(   \Phi_0 \abs w^2  +   \Phi_1 \abs{w_y}^2 \big) \, dy \, dt
\leq -  \int_{\tau_1}^{\tau_2} \big[ \phi_y \abs{w_y}^2 \big]_a^b \, dt +  \frac 1 2 \int_{\tau_1}^{\tau_2} \int_a^b \abs{g}^2 \, dy \, dt ,
\end{equation*}
where
\begin{equation} \label{def-Phi0}
\Phi_0 = - 2 \phi_y^2 \phi_{yy}  - \frac {\phi_{tt}}{2} + \frac {\phi_{yyyy}}{2} + 2\phi_{ty} \phi_{y} - \frac {n^{\frac{3}{2}}q^2 q'}{\sqrt 2} 
\end{equation}
and 
\begin{equation} \label{def-Phi1}
\Phi_1 = - 2 \phi_{yy} - \sqrt {2n} q'.
\end{equation}
\end{proposition}

\begin{proof}
We can rewrite \eqref{eq-w} as
\begin{equation*} \label{eq-g}
\big( - w_{yy} + \Phi w \big) + \big( w_t - 2  \phi_y w_y -  \phi_{yy} w + i n q^2 w \big) =  g, 
\end{equation*}
where $\Phi  = \phi_t -  \phi_y^2$. 
The identity $2\Re(\a\overline\b)\leq \abs{\a+\b}^2$ then gives, after integration,
\begin{multline} \label{ineq-w-g}
\Re \int_{\tau_1}^{\tau_2} \int_a^b \big( - w_{yy} + \Phi w \big) \big(   \overline{w_t} - 2  \phi_y \overline{w_y} - \phi_{yy} \overline{w}  - i n q^2 \overline{w} \big) \, dy \, dt\\
\leq \frac 12 \int_{\tau_1}^{\tau_2} \int_a^b |g|^2 \, dy \, dt.
\end{multline}
We estimate the left-hand side with integrations by parts, using \eqref{eq-w-Dir} and \eqref{eq-w-0T}. The terms involving $\overline{w_t}$ give
\[
\Re \int_{\tau_1}^{\tau_2}\int_a^b (-w_{yy}) \overline{w_t}  \, dy \, dt = 0
\]
and 
\[
\Re \int_{\tau_1}^{\tau_2} \int_a^b (\Phi w) \overline{w_t}  \, dy \, dt = - \frac 1 2 \int_{\tau_1}^{\tau_2} \int_a^b \Phi_t \abs{w}^2 \, dy \, dt.
\]
On the other hand, for all $t \in ]\tau_1,\tau_2[$ we have 
\begin{align*}
\Re \int_a^b (-w_{yy}) (-2  \phi_y \overline{w_{y}})  \, dy &=  \big[ \phi_y \abs{w_y}^2 \big]_a^b -   \int_a^b \phi_{yy} \abs{w_{y}}^2 \, dy,\\
\Re \int_a^b (-w_{yy}) (-  \phi_{yy} \overline{w})  \, dy &= - \int_a^b \phi_{yy} \abs{w_{y}}^2 \, dy + \frac {1} 2 \int_a^b \phi_{yyyy} \abs{w}^2 \, dy,\\
\Re \int_a^b (-w_{yy}) (-inq^2 \overline{w})  \, dy &= 2n \int_a^b  qq' \Im( w_y \overline w) \, dy,
\end{align*}
and
\begin{align*}
\Re \int_a^b (\Phi w) (-2  \phi_y \overline{w_{y}}-  \phi_{yy} \overline{w})  \, dy &=    \int_a^b \Phi_y \phi_y \abs w^2 \, dy,\\
\Re \int_a^b (\Phi w) (-inq^2 \overline{w})  \, dy &= 0.
\end{align*}
We integrate these five equalities over $t \in ]\tau_1,\tau_2[$, and then \eqref{ineq-w-g} gives 
\begin{multline*}
 \int_{\tau_1}^{\tau_2} \big[ \phi_y \abs{w_y}^2 \big]_a^b \, dt + \int_{\tau_1}^{\tau_2} \int_a^b  \left(- \frac {\Phi_t}2 + \frac { \phi_{yyyy}}2 + \Phi_y \phi_y \right) \abs w^2 \, dy \, dt  \\
 - 2 \int_{\tau_1}^{\tau_2} \int_a^b  \phi_{yy} \abs{w_y}^2 \, dy \, dt + 2n \int_{\tau_1}^{\tau_2} \int_a^b  qq' \Im(w_y \overline w)   \, dy \, dt  \leq \frac 12 \int_{\tau_1}^{\tau_2} \int_a^b |g|^2 \, dy \, dt ,
\end{multline*}
Since 
\[
2n qq' \Im(w_y \overline w) \geq - \sqrt 2 \sqrt n q' \abs{w_y}^2 - \frac { n^{\frac 32}  q^2 q' \abs w^2} {\sqrt 2},
\]
the conclusion follows.
\end{proof}

\subsection{Observability inequality for a fixed Fourier parameter} \label{sec-obs-fixed-n}

In this paragraph we prove Proposition \ref{prop-obs-fix-n} about observability for a fixed Fourier parameter $n \in \N$. As already said, this is nothing but the well-known observability inequality for a heat equation with a (complex) potential. Nevertheless, we
propose a proof here, both for the sake of self-containment, and because we believe it enlightens the following paragraph.

The proof of Proposition \ref{prop-obs-fix-n} relies on Proposition \ref{prop-Carleman}. For the time dependence of the weight $\vf$, we will use the function $\th$ given in the following lemma.

\begin{lemma} \label{lemma_def_theta}
Let $\tau_1,\tau_2 > 0$ with $\tau_1<\tau_2$. There exists $\theta$ in $C^\infty(]\tau_1,\tau_2[)$ such that
\begin{enumerate} [\rm (i)]
\item $\theta \geq 1$ on $]\tau_1,\tau_2[$, $\theta \equiv 1$ on $\big[\frac{2\tau_1 + \tau_2}{3},\frac{\tau_1 + 2\tau_2}{3} \big]$,
\item $\lim_{t \rightarrow \tau_1} \theta(t) = \lim_{t \rightarrow \tau_2} \theta(t)= + \infty$,
\item there exists a constant $C>0$ such that for all $t \in ]\tau_1,\tau_2[$,
$$
\vert \theta'(t) \vert \leq C \theta(t)^2, \quad \vert \theta''(t) \vert \leq C \theta(t)^3.
$$
\end{enumerate}
\end{lemma}

\begin{proof}
Let $\chi \in C_0^\infty \big( ]\tau_1,\tau_2[,[0,1] \big)$ be equal to 1 on $\big[\frac{2\tau_1 + \tau_2}{3},\frac{\tau_1 + 2\tau_2}{3} \big]$. For $t \in ]\tau_1,\tau_2[$ we set 
$$
\theta (t) = 1 + \frac{1 - \chi(t)}{(t-\tau_1)(\tau_2-t)}.
$$
Then $\theta$ verifies all the required properties.
\end{proof}

Now we can prove Proposition \ref{prop-obs-fix-n}:

\begin{proof}[Proof of Proposition \ref{prop-obs-fix-n}]
For $y \in \bar I$ we set 
\[
\psi(y) = \psi_1 \left( \frac {2y + \ell_- - \ell_+}{\ell_- + \ell_+} \right), \quad \text{where} \quad \psi_1(\eta) = - \frac {\eta^2} 2 \pm 2 \eta + 3, \quad \eta\in[-1,1]
\]
(the sign in front of $2\eta$ is not important here, but it has to be chosen carefully if we only observe from one side of the boundary, as will be the case in Proposition \ref{prop-carlemen-theta-psi} below).
In particular, for some $c_0 > 0$ we have on $\bar I$
\begin{equation} \label{eq-prop-psi}
\psi'' \leq - c_0, \quad \abs{\psi'} \geq c_0, \quad \psi \geq c_0.
\end{equation}

Let $u$ be a solution of \eqref{eq-Kolmogorov-n}. Let $s >1$ to be chosen large enough later. For $t \in ]\tau_1,\tau_2[$ and $y \in \bar I$ we set
$$
\phi(t,y) = s \,  \theta(t) \psi (y), 
$$
where $\theta$ is given by Lemma \ref{lemma_def_theta}, and
\[
w(t,y) = u(t,y) e^{-\phi(t,y)}.
\]
Then $w$ satisfies \eqref{eq-w}-\eqref{eq-w-0T} with $a = -\ell_-$, $b = \ell_+$ and $g \equiv 0$. Therefore,
Proposition \ref{prop-Carleman} gives
$$
\int_{\tau_1}^{\tau_2} \int_I \big(   \Phi_0 \abs w^2  +   \Phi_1 \abs{w_y}^2 \big) \, dy \, dt
\leq -  \int_{\tau_1}^{\tau_2} \big[ \phi_y \abs{w_y}^2 \big]_{-\ell_-}^{\ell_+} \, dt,
$$
with 
\begin{equation} \label{eq_Phi0_fixedn}
\Phi_0  =  s^3 \left(  - 2 \theta^3 (\psi')^2 \psi''  - \frac {\theta'' \psi}{2\,s^2} + \frac {\theta\, \psi^{(4)}}{2\,s^2} + \frac{2\theta' \theta (\psi')^2}{s}  - \frac{n^{\frac{3}{2}}q^2 q'}{s^3\sqrt{2}}  \right)
\end{equation}
and
\begin{equation} \label{eq_Phi1_fixedn}
\Phi_1 = s \left( - 2 \theta \psi''-  \frac{\sqrt {2n} q'}{s}\right).
\end{equation}
Thus, by Lemma \ref{lemma_def_theta} and \eqref{eq-prop-psi} we can fix $s$ so large that $\Phi_0 \geq 1$ and $\Phi_1 \geq 1$ on $]\tau_1,\tau_2[ \times \bar I$. This gives
$$
\int_{\tau_1}^{\tau_2} \int_{I}  \vert w(t,y) \vert^2  \, dy\, dt \lesssim \int_{\tau_1}^{\tau_2} \big( \vert w_y(t,-\ell_-) \vert^2 + \vert w_y(t,\ell_+) \vert^2 \big) \, dt,
$$
and then, since $\theta \equiv 1$ on $\big[\frac{2\tau_1 + \tau_2}{3},\frac{\tau_1 + 2\tau_2}{3} \big]$ and $\psi$ is bounded away from 0,
$$
\int_{\frac{2\tau_1 + \tau_2}{3}}^{\frac{\tau_1 + 2\tau_2}{3}} \int_{I}  \vert u(t,y) \vert^2 \, dy\, dt \lesssim \int_{\tau_1}^{\tau_2} \big( \vert u_y(t,-\ell_-) \vert^2 + \vert u_y(t,\ell_+) \vert^2 \big) \, dt.
$$
we have $\nr{u(T)}_{L^2(I)}^2 \leq \nr{u(t)}_{L^2(I)}^2$ for all $t \in \big[\frac{2\tau_1 + \tau_2}{3},\frac{\tau_1 + 2\tau_2}{3} \big]$. After integration this gives
$$
\nr{u(T)}_{L^2(I)}^2 \leq \frac{3}{\tau_2 - \tau_1} \int_{\frac{2\tau_1 + \tau_2}{3}}^{\frac{\tau_1 + 2\tau_2}{3}} \nr{u(t)}_{L^2(I)}^2 \, dt  \lesssim \int_{\tau_1}^{\tau_2} \big( \vert u_y(t,-\ell_-) \vert^2 + \vert u_y(t,\ell_+) \vert^2 \big) \, dt,
$$
which ends the proof.
\end{proof}

Notice that in this rough proof we have not tried to control the dependence of $C_n$ with respect to $n$. It is the purpose of the next paragraph to get a precise estimate of the cost of observability for \eqref{eq-Kolmogorov-n}. The interest of Proposition \ref{prop-obs-fix-n} is that it is now enough to consider only large values of $n$.

To obtain estimates in the high frequency regime, we will use the same strategy, but we will choose more carefully the parameter $s$ and the phase function $\psi$ (both should be chosen as small as possible).

From \eqref{eq_Phi0_fixedn}, we see that $s^3$ has to be at least of order $n^{\frac{3}{2}}$, while in \eqref{eq_Phi1_fixedn}, $s$
has to be of order $\sqrt n$. From these observations, we deduce that the correct scaling should be $s \sim \sqrt n$.

Finally, with $s = \sqrt n$, it is then the choice of $\psi$ that will make $\Phi_0$ and $\Phi_1$ positive for $n$ large enough.
We see from \eqref{eq_Phi0_fixedn}-\eqref{eq_Phi1_fixedn} that $\psi$ should satisfy
\begin{equation} \label{eq-hyp-psi}
-2 (\psi')^2 \psi'' - \frac{q^2q'}{\sqrt{2}} > 0 \quad \text{and} \quad  - 2\,\psi'' -  \sqrt{2} q' >0.
\end{equation}
This leads to the construction of the weight function given in the next paragraph.

\subsection{A refined Carleman estimate}

In this paragraph we prove a refined version of Proposition \ref{prop-Carleman} for $n$ large and a suitable choice for $\psi$. As discussed at the end of Paragraph \ref{sec-obs-fixed-n}, we will choose $\vf$ proportional to $\sqrt n$. The choice of $\psi$ satisfying \eqref{eq-hyp-psi} will be discussed in Proposition \ref{prop-def-phi}.

\begin{proposition} \label{prop-carlemen-theta-psi}
Let $a,b \in \bar I$ with $a < b$ and $\psi \in C^4([a,b],\R)$. We assume that for some $\e > 0$ we have on $[a,b]$:
\[
\psi \geq \e,  \quad  - 2 (\psi')^2 \psi'' - \frac {q^2 q'}{\sqrt 2} \geq \e , \quad - 2 \psi'' - \sqrt 2 q' \geq \e.
\]
Let $\tau_1,\tau_2 \in ]0,T]$ with $\tau_1<\tau_2$. For $t \in ]\tau_1,\tau_2[$ and $y \in [a,b]$ we set $\varphi(t,y) = \theta(t) \psi(y)$, where $\theta$ is given by Lemma \ref{lemma_def_theta}. Let $n \in \N$ and $u$ in
\begin{equation} \label{eq-C0-C1}
C^0\big([\tau_1,\tau_2],H^2(a,b) \cap H_0^1(a,b)\big) \cap C^1\big([\tau_1,\tau_2],L^2(a,b)\big).
\end{equation}
We set
\begin{equation} \label{eq-Kolmogorov-f}
f = u_t - u_{yy} + inq(y)^2 u,
\end{equation}
and
\begin{equation*}
w = u e^{-\sqrt n \f}, \quad g = f e^{-\sqrt n \f}.
\end{equation*}
Then there exist $N \in \N$ and $C>0$ such that the following statements hold if $n \geq N$.
\begin{enumerate}[\rm (i)]
\item If $\psi' > 0$, 
\begin{multline*} 
\int_{\tau_1}^{\tau_2} \int_a^b \big( n^{\frac 32} \th^3 \abs w^2  + \sqrt n \th \abs{w_y}^2 \big)  \, dy \, dt\\
\leq C \sqrt n \int_{\tau_1}^{\tau_2} \abs{w_y(t,a)}^2  \, dt +  C \int_{\tau_1}^{\tau_2} \int_a^b \abs g^2 \, dy \, dt.
\end{multline*}
\item If $\psi' < 0$, 
\begin{multline*} 
\int_{\tau_1}^{\tau_2} \int_a^b \big( n^{\frac 32} \th^3 \abs w^2  + \sqrt n \th \abs{w_y}^2 \big) \, dy \, dt\\
\leq C \sqrt n \int_{\tau_1}^{\tau_2} \abs{w_y(t,b)}^2  \, dt +  C \int_{\tau_1}^{\tau_2} \int_a^b \abs g^2  \, dy \, dt.
\end{multline*}
\end{enumerate}
\end{proposition}

\begin{proof} 
We observe that $\varphi$ belongs to $C^4(]\tau_1,\tau_2[ \times [a,b])$, the functions $f$ and $g$ are in $C^0([\tau_1,\tau_2], L^2(a,b))$, $w$ extends to a function in \eqref{eq-C0-C1} and we have
\begin{equation*} 
w_t - w_{yy} + inq(y)^2 w + \sqrt n \f_t w - 2 \sqrt n \f_y w_y - n \f_y^2 w - \sqrt n \f_{yy} w = g.
\end{equation*}
Moreover, $w$ satisfies the boundary conditions \eqref{eq-w-Dir} and the initial and final conditions \eqref{eq-w-0T}. Then, by Proposition \ref{prop-Carleman} applied with $\vf = \sqrt n \f$, we have   
\begin{multline*} 
\int_{\tau_1}^{\tau_2} \int_a^b \big( n^{\frac 32}  \Phi_0 \abs w^2  + \sqrt n  \Phi_1 \abs{w_y}^2 \big) \, dy \, dt\\
\leq - \sqrt n \int_{\tau_1}^{\tau_2} \big[ \f_y \abs{w_y}^2 \big]_a^b \, dt +  \frac 1 2 \int_{\tau_1}^{\tau_2} \int_a^b \abs{g}^2 \, dy \, dt ,
\end{multline*}
where
\begin{equation*} 
\Phi_0 = - 2 \f_y^2 \f_{yy} - \frac {q^2 q'}{\sqrt 2}  - \frac {\f_{tt}}{2n} + \frac {\f_{yyyy}}{2n} + \frac {2\f_{ty} \f_{y}}{\sqrt n} 
\end{equation*}
and 
\begin{equation*} 
\Phi_1 = - 2 \f_{yy} - \sqrt 2 q'.
\end{equation*}
The properties of $\th$ and the boundedness of the derivatives of $\psi$ give, for $n$ large enough,
\[
\Phi_0(t,y) \geq \frac {\e \th^3} 2 \quad \text{and} \quad \Phi_1(t,y) \geq \e \th.
\]
Thus,
\begin{multline*} 
\frac \e 2 \int_{\tau_1}^{\tau_2} \int_a^b \big( n^{\frac 32} \th^3 \abs w^2  + \sqrt n \th \abs{w_y}^2 \big) \, dy \, dt\\
\leq - \sqrt n \int_{\tau_1}^{\tau_2} \big[ \f_y \abs{w_y}^2 \big]_a^b \, dt +  \frac 12 \int_{\tau_1}^{\tau_2} \int_a^b \abs g^2 \, dy \, dt.
\end{multline*}
Notice that the assumptions on $\psi$ imply that $\psi'$ does not vanish. If $\psi'$ takes positive values then we have 
\[
 - \sqrt n \int_{\tau_1}^{\tau_2} \f_y(t,b) \abs{w_y(t,b)}^2   \, dt \leq 0,
\]
which gives the first inequality. Otherwise $\psi' < 0$ and we similarly get the second estimate.
\end{proof}

\subsection{Precise estimate of the cost of observation in small time for \texorpdfstring{$n$}{n} large} \label{sec-observability-fin}

In this paragraph we finish the proof of Proposition \ref{prop-obs-large-n}.

We could apply directly Proposition \ref{prop-carlemen-theta-psi} and observe from one side of $I$ only. However, we can reduce the cost of observability if we observe from both sides. 

More precisely, the part of $u$ in $[0,\ell_+]$ will be controled by the values of $u_y$ at $\ell_+$, and the part of $u$ in $[-\ell_-,0]$ will be controled by the values of $u_y$ at $-\ell_-$. Thus, with the notation of the previous paragraph, we have to choose $\psi$ such that $\psi' < 0$ on the right and $\psi' > 0$ on the left. Since $\psi'$ does not vanish, we have to apply Proposition \ref{prop-carlemen-theta-psi} separately on the right and on the left.

\begin{proposition} \label{prop-def-phi}
Let $\tau_1$, $\tau_2$ and $\kappa$ be as given by Proposition \ref{prop-obs-large-n}. There exist $N \in \N^*$, $\f \in C^0 \big( ]\tau_1,\tau_2[ \times \bar I,\R \big)$ and $C > 0$ such that 
\begin{equation} \label{eq-phi}
\forall t \in \left[ \frac {2\tau_1 + \tau_2} 3 , \frac {\tau_1 + 2 \tau_2}{3} \right], \forall y \in  I, \quad 0 \leq \f(t,y) \leq \k,
\end{equation}
and for any $n \geq N$ and any solution $u$ of \eqref{eq-Kolmogorov-n} we have 
\begin{multline*} \label{eq-Carleman-a}
\int_{\tau_1}^{\tau_2} \int_{I}  \big( n^{\frac 32} \abs u^2  + \sqrt n  \abs{u_y}^2 \big) e^{-2\sqrt n \f} \, dy \, dt
\\ \leq C \sqrt n \int_{\tau_1}^{\tau_2}  \big(\abs{u_y(t,-\ell_-)}^2 + \abs{u_y(t,\ell_+)}^2\big)  \, dt. 
\end{multline*}
\end{proposition}

\begin{proof}
\stepp Let $\beta > \frac 1 {\sqrt 2}$ and $\e_0 > 0$ be such that 
\begin{equation} \label{eq-eps-kappa}
\e_0 + \beta \max\left( \int_{0}^{\ell_+} \big( q(s) + 3 \e_0 \big) \, ds, \int_{-\ell_-}^0 \big( \abs{q(s)} + 3 \e_0 \big) \,ds \right) < \k.
\end{equation}
Let $\d \in ]0,\min(\ell_-,\ell_+)]$ be such that  $\max(\abs{q(-\delta)},q(\delta)) \leq \e_0$. For $y \in [-\d,\ell_+]$ we set 
\[
\psi_+(y) = \e_0 + \beta \int_{y}^{\ell_+} \big( q(s)+3\e_0 \big) \, ds + c_+,
\]
with $c_+ \geq 0$ to be chosen later.
Then we have 
\[
\psi_+ \geq \e_0, \quad \psi_+' = -\beta  (q + 3 \e_0) \leq - 2 \beta \e_0, \quad \psi_+'' = - \beta q',
\]
so 
\begin{align*}
- 2 \psi_+'^2 \psi_+'' - \frac {q^2 q'}{\sqrt 2}
&  = 2 \beta^3    (q+3\e_0)^2 q' -   \frac {q^2 q'}{\sqrt 2}\\
& \geq \frac {q'}{\sqrt 2} \big( (q+3\e_0)^2 - q^2 \big)\\
&  \geq \e_0^2 \min(q')
\end{align*}
and
\[
- 2 \psi_+'' - \sqrt 2 q' \geq 2 \left( \beta - \frac{1}{\sqrt 2}\right) q' \geq 2 \left( \beta - \frac{1}{\sqrt 2}\right)  \min(q').
\]
Thus $\psi_+$ satisfies the assumptions of Proposition \ref{prop-carlemen-theta-psi} on $[-\d,\ell_+]$.
Then, for $t \in ]\tau_1,\tau_2[$ we set
\begin{equation} \label{def-phi-+}
\f_+(t,y) = \th(t) \psi_+(y),
\end{equation}
where $\th$ is given by Lemma \ref{lemma_def_theta}.

\stepp We consider $\h_+ \in C^\infty(\bar I,[0,1])$ such that $\h_+ = 1$ on $[0,\ell_+]$ and $\h_+ = 0$ on $[-\ell_-,-\d]$. Then we set $u_+ = \h_+ u$. It satisfies 
\[
\forall t \in ]\tau_1,\tau_2[, \quad u_+(t,-\d) = u_+(t,\ell_+) = 0 
\]
and 
\[
\forall t \in ]\tau_1,\tau_2[, \forall y \in [-\d,\ell_+], \quad \big( \partial_t - \partial_{yy} + inq^2 \big) u_+(t,y) = f_+(t,y),
\]
where 
\[
f_+ = - \h_+'' u  - 2 \h_+' u_y.
\]
In particular, $f_+(t,\cdot)$ is supported in $[-\d,0]$. We set 
\[
w_+ = u_+ e^{-\sqrt n \f_+} 
\quad \text{and} \quad 
g_+ = f_+ e^{-\sqrt n \f}.
\]
We have
\begin{align*}
\sqrt n \abs{\partial_y u_+}^2 e^{-2 \sqrt n \f_+}
& \lesssim \sqrt n \abs{\partial_y w_+}^2 + n^{\frac 32} \abs{w_+}^2 \th(t)^2 . 
\end{align*}
Then, by the second case in Proposition \ref{prop-carlemen-theta-psi}, we obtain 
\begin{eqnarray}\label{eq++}
\lefteqn{\int_{\tau_1}^{\tau_2} \int_0^{\ell_+} \big( n^{\frac 32} \abs {u_+}^2  + \sqrt n   \abs{\partial_y u_+}^2 \big) e^{- 2 \sqrt n \f_+} \, dy \, dt}\\
\nonumber
&& \lesssim \int_{\tau_1}^{\tau_2} \int_0^{\ell_+} \big( n^{\frac 32} \th^2 \abs {w_+}^2  + \sqrt n   \abs{\partial_y w_+}^2 \big) \, dy \, dt\\
\nonumber
&& \lesssim \sqrt n \int_{\tau_1}^{\tau_2} \abs{\partial_y w_+(t,\ell_+)}^2  \, dt +  \int_{\tau_1}^{\tau_2} \int_{-\d}^0 \abs {g_+}^2 \, dy \, dt\\
\nonumber
&& \lesssim \sqrt n \int_{\tau_1}^{\tau_2} \abs{\partial_y u_+(t,\ell_+)}^2  \, dt +  \int_{\tau_1}^{\tau_2} \int_{-\d}^0 \abs {f_+}^2 e^{-2\sqrt n \f_+} \, dy \, dt.
\end{eqnarray}

\stepp 
For $y \in [-\ell_-,\d]$ we set 
$$
\psi_-(y) = \e_0 + \beta \int_{-\ell_-}^y (\abs{q(s)} + 3 \varepsilon_0)\, ds + c_-,
$$
with $c_- \geq 0$ to be chosen later, and for $t \in ]0,T[$,
\[
\f_-(t,y) = \theta(t) \psi_-(y).
\]
Let  $\h_- \in C^\infty([-\ell_-,\ell_+],[0,1])$ such that $\h_- = 1$ on $[-\ell_-,0]$ and $\h_- = 0$ on $y \in [\d,\ell_+]$. We set  $u_- = \h_- u$ and $f_- = -\h_-'' u - 2 \h_-' u_y$. Then, as above, but using the first statement in Proposition \ref{prop-carlemen-theta-psi}, we obtain
\begin{multline} \label{eq--}
\int_{\tau_1}^{\tau_2} \int_{-\ell_-}^0 \big( n^{\frac 32}  \abs {u_-}^2  + \sqrt n  \abs{\partial_y u_-}^2 \big) e^{- 2 \sqrt n \f_-} \, dy \, dt\\
\lesssim \sqrt n \int_{\tau_1}^{\tau_2} \abs{\partial_y u_-(t,-\ell_-)}^2  \, dt + \int_{\tau_1}^{\tau_2} \int_0^\d \abs {f_-}^2 e^{-2\sqrt n \f_-} \, dy \, dt.
\end{multline}

\stepp We set $c_+ = \max(0,c)$ and $c_- = \max(0,-c)$ where
$$
c = \beta \left(\int_{-\ell_-}^0 (\abs{q(s)} + 3\e_0)\, ds - \int_0^{\ell_+}(q(s) + 3\, \e_0)\,ds  \right),
$$
so that $\psi_+(0) = \psi_-(0)$. Then for $t \in ]\tau_1,\tau_2[$ and $y \in \bar I$ we set 
\[
\f(t,y) = 
\begin{cases}
\f_-(t,y) & \text{if } y \leq 0,\\
\f_+(t,y) & \text{if } y \geq 0.
\end{cases}
\]
In particular, by construction, $\f$ is continuous on $]\tau_1,\tau_2[\times \bar I$ and satisfies  \eqref{eq-phi}. Moreover, $\f_+ \geq \f$ on $[-\d,0]$, $\f_- \geq \f$ on $[0,\d]$ and, on $[-\d,\d]$,
\[
\abs{f_+} + \abs {f_-} \lesssim \abs {u} + \abs{u_y}.
\]
Then, by summing \eqref{eq++} and \eqref{eq--},
\begin{multline*} 
\int_{\tau_1}^{\tau_2} \int_{I} \big( n^{\frac 32}  \abs u^2  + \sqrt n  \abs{u_y}^2 \big) e^{-2 \sqrt n \f} \, dy \, dt\\
\lesssim \sqrt n \int_{\tau_1}^{\tau_2} \big( \abs{u_y(t,-\ell_-)}^2 + \abs{u_y(t,\ell_+)}^2 \big)  \, dt + \int_{\tau_1}^{\tau_2} \int_{-\d}^\d \big(\abs u^2 + \abs{u_y}^2\big) e^{-2\sqrt n \f} \, dy \, dt.
\end{multline*}
For $n$ large enough, the last term is smaller than one half of the left-hand side, and the conclusion follows.
\end{proof}

We can now prove Proposition \ref{prop-obs-large-n}.

\begin{proof}[Proof of Proposition \ref{prop-obs-large-n}] 
Let $N$ be given by Proposition \ref{prop-def-phi} and $n \geq N$. Let $u$ be a solution of \eqref{eq-Kolmogorov-n}. Let $\f$ be given by Proposition \ref{prop-def-phi}. By \eqref{eq-phi} we have in particular
\[
\int_{\frac{2 \tau_1 + \tau_2}{3}}^{\frac{\tau_1 + 2\tau_2}{3}} \int_{I} \abs{u}^2  \, dy \, dt \leq \frac {C e^{2 \k \sqrt n}} n \int_{\tau_1}^{\tau_2} \big( \abs{u_y (t,-\ell_-)}^2 + \abs{u_y(t,\ell_+)}^2 \big) \, dt.
\]
By \eqref{eq-semigroup-contraction} we have $\nr{u(\tau_2)}_{L^2(I)}^2 \leq \nr{u(t)}_{L^2(I)}^2$ for all $t \in \left] \frac{2\tau_1 + \tau_2}{3}, \frac{\tau_1 + 2\tau_2}{3} \right[$, so 
\[
\nr{u(\tau_2)}_{L^2(I)}^2 \leq  \frac{3\,C e^{2 \k \sqrt n}}{(\tau_2-\tau_1)n} \int_{\tau_1}^{\tau_2} \big( \abs{u_y (t,-\ell_-)}^2 + \abs{u_y(t,\ell_+)}^2 \big) \, dt,
\]
and Proposition \ref{prop-obs-large-n} is proved.
\end{proof}

\subsection*{Aknowledgements} We express our gratitude to Karine Beauchard, for enriching discussions on this work. This work has been supported by the CIMI Labex, Toulouse, France, under grant ANR-11-LABX-0040-CIMI.

\end{document}